\documentclass[10pt]{amsart}

\usepackage{setspace}
\makeatletter
\usepackage{xstring}
\renewcommand{\tocappendix}[3]{%
  \indentlabel{\IfStrEq{#3}{Bibliography}{}{#1}\@ifnotempty{#2}{ #2.\quad}}#3}
\makeatother

\usepackage{palatino}

\usepackage{macros}
\standardsettings
\let\O\undefined
\DeclareMathOperator{\O}{O}
\DeclareMathOperator{\Aut}{Aut}

\newcommand{\bp}{p}
\newcommand{\bq}{q}
\newcommand{\dox}[1]{\|#1\|}

\newcommand{\dogo}[1]{\|#1\|}

\numberwithin{figure}{section}% There seems to be a bug with doing anything other than this.

\theoremstyle{definition}
\maketheorems{standingassumption}{Standing Assumption}
\maketheorems{edelsteinexample}{Edelstein-type Example}

\draftfalse

\begin{document}

%\title[Geometry and dynamics in hyperbolic metric spaces]{Geometry and dynamics in Gromov hyperbolic metric spaces I \\ {\small{\textnormal{With an emphasis on non-proper settings}}}}

\title[Tukia's isomorphism theorem in CAT($-1$) spaces]{
\uppercase{
\large Tukia's isomorphism theorem in CAT($-1$) spaces
}
}



\authortushar\authordavid\authormariusz

\begin{Abstract}
We prove a generalization of Tukia's ('85) isomorphism theorem, which states that any isomorphism between two geometrically finite groups extends equivariantly to a quasisymmetric homeomorphism between their limit sets. Tukia worked in the setting of real hyperbolic spaces of finite dimension, and his theorem cannot be generalized as stated to the setting of CAT($-1$) spaces. We exhibit examples of type-preserving isomorphisms of geometrically finite subgroups of finite-dimensional rank one symmetric spaces of noncompact type (ROSSONCTs) whose boundary extensions are not quasisymmetric. A sufficient condition for a type-preserving isomorphism to extend to a quasisymmetric equivariant homeomorphism between limit sets is that one of the groups in question is a lattice, and that the underlying base fields are the same, or if they are not the same then the base field of the space on which the lattice acts has the larger dimension. This in turn leads to a generalization of a rigidity theorem of Xie ('08) to the setting of finite-dimensional ROSSONCTs.
\end{Abstract}

\maketitle

%\tableofcontents

%%%%%%%%%%%%%%%%%%%%%%%%%%%%%%%%%%
\draftnewpage
\section{Introduction}
%%%%%%%%%%%%%%%%%%%%%%%%%%%%%%%%%%

Tukia's isomorphism theorem \cite[Theorem 3.3]{Tukia2} states that any type-preserving isomorphism $\Phi$ between two geometrically finite subgroups of $\Isom(\H^d)$ (not necessarily the same $d$ for both groups) extends to a quasisymmetric equivariant homeomorphism between their limit sets. In this note we prove a generalization of this theorem to the setting of CAT(-1) spaces, paying particular attention to the case of rank one symmetric spaces of noncompact type (ROSSONCTs). However, our theorem cannot be stated as the naive (word-for-word) generalization of Tukia's theorem, since such a generalization is false (cf. Example \ref{exampletukia} and Remark \ref{remarktukia}). Instead, we establish sufficient conditions on a type-preserving isomorphism $\Phi$ between two geometrically finite subgroups of $\Isom(X)$, where $X$ is a CAT(-1) space, in order for the conclusion of Tukia's theorem to hold.

\begin{convention}
\label{conventionimplied}
The symbols $\lesssim$, $\gtrsim$, and $\asymp$ will denote asymptotics; a subscript of $\plus$ indicates that the asymptotic is additive, and a subscript of $\times$ indicates that it is multiplicative. For example, $A\lesssim_{\times,K} B$ means that there exists a constant $C > 0$ (the \emph{implied constant}\label{pageimpliedconstant}), depending only on $K$, such that $A\leq C B$. Moreover, $A\lesssim_{\plus,\times}B$ means that there exist constants $C_1,C_2 > 0$ so that $A\leq C_1 B + C_2$.
%In general, dependence of the implied constant(s) on universal objects such as the metric space $X$, the group $G$, and the distinguished point $\zero\in X$ (cf. Notation \ref{standingassumptions}) will be omitted from the notation.
\end{convention}

\begin{convention}
The symbol $\triangleleft$ will be used to indicate the end of a nested proof.
\end{convention}

\begin{convention}
Given a distinguished point $\zero\in X$, we write
\begin{align*}
\dox x = \dist(\zero,x) \text{ and } \dogo g = \dox{g(\zero)}.
\end{align*}
\end{convention}

\begin{convention}
We denote the maximum of two numbers $A,B$ by $A\vee B$.
\end{convention}

{\bf Acknowledgements.} The research of the first-named author was supported in part by a 2014-2015 Faculty Research Grant from the University of Wisconsin--La Crosse. The research of the second-named author was supported in part by the EPSRC Programme Grant EP/J018260/1. The research of the third-named author was supported in part by the NSF grant DMS-1361677. The authors thank Marc Bourdon, Pierre Py, and the anonymous referee for helpful remarks and suggestions.

\subsection{Geometrically finite groups}
It was known for a long time that every finitely generated Fuchsian group has a finite-sided convex fundamental domain (e.g. \cite[Theorem 4.6.1]{Katok_book}). This result does not generalize beyond two dimensions (e.g. \cite{Bers, Jorgensen}), but subgroups of $\Isom(\H^3)$ with finite-sided fundamental domains came to be known as \emph{geometrically finite} groups. Several equivalent definitions of geometrical finiteness in the three-dimensional setting became known, for example Beardon and Maskit's condition that the limit set is the union of the radial limit set $\Lr$ with the set $\Lbp$ of bounded parabolic points \cite{BeardonMaskit}, but the situation in higher dimensions was somewhat murky until Bowditch \cite{Bowditch_geometrical_finiteness} wrote a paper that described which equivalences remain true in higher dimensions, and which do not. The condition of a finite-sided convex fundamental domain is no longer equivalent to any other conditions in higher dimensions (e.g. \cite{Apanasov}), so a higher-dimensional Kleinian group is said to be \emph{geometrically finite} if it satisfies any of Bowditch's five equivalent conditions (GF1)-(GF5). Four of Bowditch's equivalent conditions make sense in the more general setting of pinched Hadamard manifolds \cite{Bowditch_geometrical_finiteness2}, but when the generality is extended further to CAT(-1) spaces, only conditions (GF1) and (GF2) remain equivalent (cf. \cite[Remark 12.4.6]{DSU}). Groups acting on proper CAT(-1) spaces and satisfying equivalent conditions (GF1) and (GF2) were first called geometrically finite by Roblin \cite[\61F]{Roblin1}; the authors of the current paper removed the properness assumption in  \cite[\612]{DSU}, where we gave the following definition:

\begin{definition}[{\cite[Definition 12.4.1]{DSU}}; cf. {\cite[Proposition 1.10(iii)]{Roblin1}}]
\label{definitionGF}
Let $X$ be a CAT(-1) space. A group $G\leq\Isom(X)$ is said to be \emph{strongly discrete} if for all $\rho > 0$,
\[
\#\{g\in G : \|g\| \leq \rho\} < \infty.
\]
A strongly discrete group $G$ is said to be \emph{geometrically finite} if there exists a disjoint $G$-invariant collection of horoballs $\scrH$ and a basepoint $\zero\notin\bigcup\scrH$ such that
\begin{itemize}
\item[(I)] for every $\rho > 0$, the set
\[
\scrH_\rho := \{ H\in\scrH : \dist(\zero,H) \leq \rho \}
\]
is finite, and
\item[(II)] there exists $\sigma > 0$ such that
\[
\CC_\zero \subset G(B(\zero,\sigma)) \cup \bigcup\scrH,
\]
where $\CC_\zero$ denotes the quasiconvex core of $G$ (with respect to the basepoint $\zero\in X$); cf. Definition \ref{definitionconvexhull}.
\end{itemize}
\end{definition}

\begin{remark*}
The notion of geometrical finiteness is closely linked with the notion of \emph{relative hyperbolicity}, a concept introduced by Gromov in \cite{Gromov3} and subsequently studied intensively by the geometric group theory community; for recent advances see \cite{Bowditch_relatively_hyperbolic, Osin2, Yaman}. To be precise, if $G$ is a geometrically finite group whose maximal parabolic subgroups are finitely generated,\Footnote{It is possible that the arguments of \cite{Osin2} might be able to remove the hypothesis that the maximal parabolic subgroups are finitely generated. In any case, this hypothesis is satisfied when the space in question is a finite-dimensional ROSSONCT, which is our main example in this paper.\label{footnoteosin}} then $G$ is hyperbolic relative to the collection $\{\Stab(G;p) : p\in\Lbp(G)\}$, where $\Lbp(G)$ denotes the set of bounded parabolic points of $G$ and $\Stab(G;p)$ denotes the stabilizer of $p$ in $G$ (cf. Lemma \ref{lemmarelativelyhyperbolic}).
\end{remark*}

\subsection{Main results}
We now introduce the terminology necessary to state our main theorem.

\begin{definition}
\label{definitiontypepreserving}
An isomorphism between two groups acting on CAT(-1) spaces is \emph{type-preserving} if the image of a loxodromic (resp. parabolic, elliptic) isometry is loxodromic (resp. parabolic, elliptic). (For the definitions of loxodromic, parabolic, and elliptic isometries, see Definition \ref{definitionclassification1} below.)
\end{definition}

\begin{definition}
\label{definitionquasisymmetric}
Let $(Z,\Dist)$ and $(\w Z,\w\Dist)$ be metric spaces. A homeomorphism $\phi:Z\to\w Z$ is said to be \emph{quasisymmetric} if there exists an increasing homeomorphism $f:(0,\infty)\to(0,\infty)$ such that
\[
\frac{\w\Dist(\phi(z),\phi(y))}{\w\Dist(\phi(z),\phi(x))} \leq f\left(\frac{\Dist(z,y)}{\Dist(z,x)}\right) \all x,y,z\in Z.
\]
\end{definition}

\begin{theorem}[Generalization of Tukia's isomorphism theorem]
\label{theoremtukia}
Let $X$, $\w X$ be CAT(-1) spaces,\Footnote{Or more generally, regularly geodesic strongly hyperbolic metric spaces; cf. \cite[Definitions 3.3.6 and 4.4.5]{DSU}.} let $G\leq\Isom(X)$ and $\w G\leq\Isom(\w X)$ be two geometrically finite groups, and let $\Phi:G\to \w G$ be a type-preserving isomorphism. Let $P$ be a complete set of inequivalent parabolic points for $G$.
\begin{itemize}
\item[(i)] If for every $\bp\in P$ we have
\begin{equation}
\label{tukia}
\dogo{\Phi(h)} \asymp_{\plus,\times,\bp} \dogo h \all h\in G_\bp,
\end{equation}
then there is an equivariant homeomorphism between $\Lambda := \Lambda(G)$ and $\w\Lambda := \Lambda(\w G)$.
\item[(ii)] If for every $\bp\in P$ there exists $\alpha_\bp > 0$ such that
\begin{equation}
\label{tukia2}
\dogo{\Phi(h)}  \asymp_{\plus,\bp} \alpha_\bp \dogo h \all h\in G_\bp,
\end{equation}
then the homeomorphism of \text{(i)} is quasisymmetric.
%The equivariant homeomorphism of \text{(i)} is quasisymmetric if and only if there exists a function $f_\bp:\Rplus\to\Rplus$ such that
%\begin{repequation}{tukia2}
%\|\Phi(h)\|  \asymp_{\plus,\bp} f_\bp(\ell_0(h)) \all h\in G_\bp
%\end{repequation}
%and there exist constants $0 < c < C < \infty$ such that
%\[
%f(\rho_1) + c\rho_2 + \leq f_\bp(\rho_1 + \rho_2) \leq f(\rho_1) + C\rho_2 \all \rho_1,\rho_2 > 0.
%\]
\end{itemize}
\end{theorem}

\begin{remark*}
It was pointed out to us by the referee that if the maximal parabolic subgroups of $G$ are assumed to be finitely generated, then part (i) can be proven using results of Yaman \cite{Yaman}, in a way that does not require the use of assumption \eqref{tukia}. (For details see Theorem \ref{theoremtukiaappendix}.) Since these arguments are quite different from ours, this means that there are two independent proofs of Theorem \ref{theoremtukia}(i) in this case.

In any case, the quasisymmetry result of Theorem \ref{theoremtukia}(ii) appears to be new. It is this result which properly generalizes Tukia's isomorphism theorem, which is also a quasisymmetry result. Moreover, having a quasisymmetry rather than just a homeomorphism is important for certain applications, e.g. when proving rigidity theorems; cf. the remarks preceding Example \ref{exampletukia}.
\end{remark*}

Tukia's isomorphism theorem \cite[Theorem 3.3]{Tukia2} corresponds to the case of Theorem \ref{theoremtukia} where $X$ and $\w X$ are finite-dimensional real ROSSONCTs. Note that in this case, the hypothesis \eqref{tukia2} always holds with $\alpha_\bp = 1$ (Corollary \ref{corollaryrealROSSONCTquasi}; see also \cite[Theorem 5.4.3]{Ratcliffe}). This is why Tukia's original theorem does not need to mention the condition \eqref{tukia2}.
%, due to the structure theorem for parabolic subgroups of $\Isom(X)$ \cite[Theorem 5.4.3]{Ratcliffe}.

It is natural to ask in what circumstances the assumptions \eqref{tukia} and/or \eqref{tukia2} hold. In the case of finite-dimensional nonreal ROSSONCTs, this question is partially answered by the following theorem:

\begin{theorem}
\label{theoremtukiaROSSONCT}
Let $X$ and $\w X$ be finite-dimensional ROSSONCTs, let $G\leq\Isom(X)$ and $\w G\leq\Isom(\w X)$ be geometrically finite groups, and let $\Phi:G\to \w G$ be a type-preserving isomorphism. Then \eqref{tukia} holds. Moreover, suppose that
\begin{itemize}
\item[(I)] $G$ is a lattice, and
\item[(II)] if $\F$ and $\w\F$ are the underlying base fields of $X$ and $\w X$, respectively, then $\dim_\R(\F) \geq \dim_\R(\w\F)$.
\end{itemize}
Then \eqref{tukia2} holds.
\end{theorem}

If the assumptions (I)-(II) are omitted, it is easy to construct examples of groups $G$, $\w G$ satisfying the hypotheses of Theorem \ref{theoremtukia}(i) but for which the equivariant homeomorphism is not quasisymmetric (Example \ref{exampletukia} and Remark \ref{remarktukia}). This shows that the assumption \eqref{tukia2} cannot be omitted from part (ii) of Theorem \ref{theoremtukia}.

\begin{remark*}
As remarked earlier, for groups whose maximal parabolic subgroups are finitely generated, the assumption \eqref{tukia} can in fact be omitted from part (i) of Theorem \ref{theoremtukia} via a result of Yaman. However, the first part of Theorem \ref{theoremtukiaROSSONCT} is still significant in that it implies that Theorem \ref{theoremtukia} is sufficient to deduce the existence of equivariant boundary extensions for groups acting on finite-dimensional ROSSONCTS, and Yaman's theorem is not needed.
\end{remark*}

Using Theorems \ref{theoremtukia} and \ref{theoremtukiaROSSONCT}, we generalize a rigidity theorem of Xie \cite[Theorem 3.1]{Xie} to the setting of finite-dimensional ROSSONCTs:

\begin{theorem}
\label{theoremxie}
Let $X$, $\w X$ be finite-dimensional ROSSONCTs whose base fields $\F$ and $\w\F$ satisfy $\dim_\R(\F) \geq \dim_\R(\w\F)$, with $X\neq\H_\R^2$. Let $G\leq\Isom(X)$ be a noncompact lattice, and let $\w G\leq\Isom(\w X)$ be a geometrically finite group, both torsion-free. Let $\Phi:G\to \w G$ be a type-preserving isomorphism. Then $\HD(\Lambda(\w G)) \geq \HD(\Lambda(G)) = \HD(\del X)$, where $\HD$ denotes Hausdorff dimension. Furthermore, equality holds if and only if $\w G$ stabilizes an isometric copy of $X$ in $\w X$.
\end{theorem}

\begin{remark}
The case of Theorem \ref{theoremxie} that occurs when $\F$ is the quaternions and $\w\F\in\{\R,\C\}$ is in fact vacuously true, as pointed out to us by Pierre Py. Indeed, in this case $G$ must have property (T) \cite[Theorems 1.5.3 and 1.7.1]{BHV} while $\w G$ must have the Haagerup property \cite[Theorem 4.2.15]{CCJJV}, and thus they cannot be isomorphic.
\end{remark}

We end this introduction by stating an open problem:

\begin{problem}
Is Theorem \ref{theoremxie} true if we drop the assumption $\dim_\R(\F)\geq \dim_\R(\w\F)$?
\end{problem}

{\bf Outline.} In Section \ref{sectionbackground}, we provide some background on ROSSONCTs, CAT(-1) spaces, and their geometry. In Section \ref{sectiontukia}, we prove Theorem \ref{theoremtukia}, and in Section \ref{sectiontukiaROSSONCT} we prove Theorems \ref{theoremtukiaROSSONCT} and \ref{theoremxie}. In the Appendix, we give the proofs of some assertions suggested by the referee, relating relative hyperbolicity with geometrical finiteness.

%%%%%%%%%%%%%%%%%%%%%%%%%%%%%%%%%%
\draftnewpage
\section{Background}
\label{sectionbackground}
%%%%%%%%%%%%%%%%%%%%%%%%%%%%%%%%%%

\subsection{ROSSONCTs}
\label{sectionROSSONCTs}
%%%%%%%%%%%%%%%%%%%%%%%%%%%%%%%%%%

A canonical class of negatively curved manifolds is the rank one symmetric spaces of noncompact type (ROSSONCTs), which come in four flavors, corresponding to the classical division algebras $\R$, $\C$, $\Q$ (quaternions), and $\mathbb O$ (octonions).\Footnote{We denote the quaternions by $\mathbb Q$ in order to avoid confusion with the ROSSONCT\ itself, which we will denote by $\mathbb H$. $\mathbb Q$ should not be confused with the set of rational numbers.} The first three division algebras have corresponding ROSSONCTs\ of arbitrary dimension, but there is only one ROSSONCT\ corresponding to the octonions; it occurs in dimension two (which corresponds to real dimension 16). 
%Consequently, the octonion ROSSONCT\ (known as the \emph{Cayley hyperbolic plane}\Footnote{Not to be confused with the \emph{Cayley plane}, a different mathematical object.}) does not have an infinite-dimensional analogue, while the other three classes do admit infinite-dimensional analogues. 
The ROSSONCTs\ corresponding to $\R$ have constant negative curvature, however those corresponding to the other division algebras have variable negative curvature \cite[Lemmas 2.3, 2.7, 2.11]{Quint} (see also \cite[Corollary of Proposition 4]{Heintze}). Some references for the theory of ROSSONCTs are \cite{BridsonHaefliger,CFKP,MackayTyson}, and we use the same notation as \cite[\62]{DSU}.

\begin{remark}
\label{remarkH2O}
In the remainder of the text we use the term ``ROSSONCT'' to refer to all ROSSONCTs \emph{except} the the octonion ROSSONCT\ $\H_{\mathbb O}^2$, known as the \emph{Cayley hyperbolic plane},\Footnote{Not to be confused with the \emph{Cayley plane}, a different mathematical object.} in order to avoid dealing with the complicated algebra of the octonion ROSSONCT.\Footnote{The complications come from the fact that the octonions are not associative, thus making it somewhat unclear what it means to say that $\mathbb O^3$ is a vector space ``over'' the octonions, since in general $(\xx a) b \neq \xx (a b)$.} 
However, it may be of interest to investigate whether our results generalize to include the Cayley hyperbolic plane (possibly after modifying the statements slightly). We leave this task to an algebraist. 
%For the reader interested in learning more about the Cayley hyperbolic plane, see \cite[pp.136-139]{Mostow}, \cite{SpringerVeldkamp}, or \cite{Allcock}; see also \cite{Baez} for an excellent introduction to octonions in general.
\end{remark}

Fix $\F\in\{\R,\C,\Q\}$ and $d\in\N$. Let us construct a ROSSONCT\ of type $\F$ in dimension $d$. In what follows we think of $\F^{d + 1}$ as a right $\F$-module (i.e. scalars always act on the right).\Footnote{The advantage of this convention is that it allows matrices to act on the left.} Consider the skew-symmetric sesquilinear form $B_\QQ:\F^{d + 1}\times\F^{d + 1}\to\F$ defined by
\[
B_\QQ(\xx,\yy) := -\wbar x_0 y_0 + \sum_{i=1}^{d} \wbar x_i y_i
\]
and its associated quadratic form
\begin{equation}
\label{Qdef}
\QQ(\xx) := B_\QQ(\xx,\xx) = -|x_0|^2 + \sum_{i=1}^{d} |x_i|^2.
\end{equation}
Let $\proj(\F^{d + 1})$ denote the \emph{projectivization} of $\F^{d + 1}$, i.e. the quotient of $\F^{d + 1}\butnot\{\0\}$ under the equivalence relation $\xx\sim \xx a$ ($\xx\in\F^{d + 1}\butnot\{\0\}$, $a\in \F\butnot\{0\}$). Let
\[
\H = \H_\F^d := \{ [\xx]\in\proj(\F^{d + 1}) : \QQ(\xx) < 0 \},
\]
and consider the map $\dist_\H: \H\times\H\to \Rplus$ defined by the equation
\begin{equation}
\label{distanceinL}
\cosh\dist_\H([\xx],[\yy]) = \frac{|B_\QQ(\xx,\yy)|}{\sqrt{|\QQ(\xx)|\cdot|\QQ(\yy)}|}, \;\;\; [\xx],[\yy]\in \H.
\end{equation}
The map $\dist_\H$ defines a metric on $\H$ that is compatible with the natural topology (as a subspace of the quotient space $\proj(\F^{d + 1})$). Moreover, for any two distinct points $[\xx],[\yy]\in\HH$ there exists a unique isometric embedding $\gamma:\R\to \H$ such that $\gamma(0) = [\xx]$ and $\gamma\circ\dist_\H([\xx],[\yy]) = [\yy]$.

\begin{definition}
\label{definitionROSSONCT}
A \emph{rank one symmetric space of noncompact type (ROSSONCT)} of dimension $d$ is a pair $(\H_\F^d,\dist_\H)$, where $\F\in\{\R,\C,\Q\}$ and $d\in\N$.
\end{definition}

This definition can be extended to infinite dimensions; cf. \cite[Definition 2.2.6]{DSU}.

\begin{remark}
\label{remarkROSSONCTterminology}
Technically, the above definition should really be a theorem (modulo the Cayley hyperbolic plane, cf. Remark \ref{remarkH2O}), since symmetric spaces are a certain type of Riemannian manifolds, with rank and type being properties of those manifolds; the classification of rank one symmetric spaces of noncompact type then follows from the classification of symmetric spaces generally (e.g. \cite[p.518]{Helgason}).\Footnote{In the notation of \cite{Helgason}, the spaces $\mathbb H_{\mathbb R}^p$, $\mathbb H_{\mathbb C}^p$, $\mathbb H_{\mathbb Q}^p$, and $\mathbb H_{\mathbb O}^2$ are written as $\SO(p,1)/\SO(p)$, $\SU(p,1)/\SU(p)$, $\Sp(p,1)/\Sp(p)$, and $(\mathfrak f_{4(-20)},\so(9))$, respectively.} However, we prefer to define ROSSONCTs without dealing with the algebra behind symmetric spaces in general.
\end{remark}

\subsection{Negatively curved metric spaces} \label{sectiongeometry1}
%%%%%%%%%%%%%%%%%%%%%%%%%%%%%%%%

A good reference for the theory of ``negative curvature'' in general metric spaces is \cite{BridsonHaefliger}. 
We assume that that the reader is aware of the definition of a CAT(-1) spaces, viz. these are geodesic metric spaces whose triangles are ``thinner'' than the corresponding triangles in two-dimensional real hyperbolic space $\H^2$, see \cite[p.158]{BridsonHaefliger} for details. It follows from their definition that CAT(-1) spaces are uniquely geodesic; we denote the unique geodesic segment connecting two points $x,y$ by $\geo xy$. Any Riemannian manifold with sectional curvature bounded above by $-1$ is a CAT(-1) space. Since ROSSONCTs are Riemannian manifolds with sectional curvature bounded between $-4$ and $-1$, every ROSSONCT\ is a CAT(-1) space.

The next level of generality considers Gromov hyperbolic metric spaces. These are spaces that are ``approximately $\R$-trees''. A good reference for the basics of the theory is \cite{Vaisala}. 

\begin{definition}
\label{definitiongromovhyperbolic}
A metric space $X$ is called \emph{hyperbolic} (or \emph{Gromov hyperbolic}) if for every four points $x,y,z,w\in X$ we have 
\begin{equation}
\label{gromov}
\lb x|z\rb_w \gtrsim_\plus \min(\lb x|y\rb_w,\lb y|z\rb_w),
\end{equation}
where the expression 
\begin{equation}
\label{gromovproduct}
\lb b|c\rb_a := \frac12[\dist(a,b) + \dist(a,c) - \dist(b,c)]
\end{equation}
is called the \emph{Gromov product} of $b$ and $c$ with respect to $a$. We refer to \eqref{gromov} as \emph{Gromov's inequality}.
\end{definition}

The \emph{Gromov boundary} of $X$, denoted $\del X$, is the set of Gromov sequences modulo equivalence, see \cite[Definition 3.4.1]{DSU} for details. The \emph{Gromov closure} or \emph{bordification} of $X$ is the disjoint union $\bord X := X\cup\del X$. The Gromov product can be extended in a near-continuous way to $\bord X$, see \cite[Definition 3.4.9, Lemma 3.4.22]{DSU}. For each $z\in \bord X$, let $\busemann_z$ denote the \emph{Busemann function}
\begin{equation}
\label{busemanndef}
\busemann_z(x,y) := \lb y|z \rb_x - \lb x|z \rb_y~.
\end{equation}
Note that for $z \in X$, this formula reduces to $\busemann_z(x,y) = \dist(z,x) - \dist(z,y)$.

If $X$ is a CAT(-1) space, then unique geodesicity extends to the bordification in the sense that for all $x,y\in\bord X$ such that $x\neq y$, there exists a unique geodesic segment connecting $x$ with $y$ \cite[Proposition 4.4.4]{DSU}. Again, we denote this geodesic segment by $\geo xy$.

If $X$ is a geodesic metric space, then the condition of hyperbolicity can be reformulated in several different ways, including the \emph{thin triangles condition}.
\begin{proposition}[{\cite[\6III.H.1]{BridsonHaefliger}}]
\label{propositionrips}
Assume that $X$ is a geodesic hyperbolic metric space. Let $\geo xy$ denote the geodesic segment connecting two points $x,y\in X$.
\begin{itemize}
\item[(i)] For all $x,y,z\in X$,
\[
\dist(z,\geo xy) \asymp_\plus \lb x|y\rb_z.
\]
\item[(ii)] (Rips' thin triangles condition) For all $x,y_1,y_2\in X$ and $z\in\geo{y_1}{y_2}$, we have
\[
\min_{i = 1}^2 \dist(z,\geo x{y_i}) \asymp_\plus 0.
\]
\end{itemize}
\end{proposition}
\noindent In fact, the thin triangles condition is equivalent to hyperbolicity; see e.g. \cite[Proposition III.H.1.22]{BridsonHaefliger}.

\begin{convention}
In the remainder of this text, $X$ denotes a CAT(-1) space, and $\zero\in X$ denotes a distinguished point.
\end{convention}

In particular, $X$ is a geodesic Gromov hyperbolic metric space. However, $X$ is not necessarily proper. We keep in mind the special case where $X$ is a ROSSONCT (finite- or infinite-dimensional).

We now recall various definitions and theorems from \cite{DSU}.

%%%%%%%%%%%%%%%%%%%%%%%%%%%%%%%%
\subsection{Classification of isometries}
\label{sectionclassification}
%%%%%%%%%%%%%%%%%%%%%%%%%%%%%%%%

For $g\in\Isom(X)$, let $\Fix(g) := \{x\in\bord X:g(x) = x\}$. Given $\xi\in\Fix(g)\cap\del X$, $\xi$ is a \emph{neutral} or \emph{indifferent} fixed point if $g'(\xi) = 1$,
an \emph{attracting} fixed point if $g'(\xi) < 1$, and a \emph{repelling} fixed point if $g'(\xi) > 1$. Here, $g'(\xi)$ denotes the metric derivative of $g$ at $\xi$ (see \cite[\64.2.2]{DSU}).

\begin{definition}
\label{definitionclassification1}
An isometry $g\in\Isom(X)$ is called
\begin{itemize}
\item \emph{elliptic} if the orbit $\{g^n(\zero):n\in\N\}$ is bounded,
\item \emph{parabolic} if it is not elliptic and has a unique fixed point, which is neutral, and
\item \emph{loxodromic} if it has exactly two fixed points, one of which is attracting and the other of which is repelling.
\end{itemize}
\end{definition}

The categories of elliptic, parabolic, and loxodromic are clearly mutually exclusive. Conversely, any isometry is either elliptic, parabolic, or loxodromic (e.g. \cite[Theorem 6.1.4]{DSU}).

%%%%%%%%%%%%%%%%%%%%%%%%%%%%%%%%%%
\bigskip
\subsection{The limit set}
%%%%%%%%%%%%%%%%%%%%%%%%%%%%%%%%%%

An important invariant of a group $G\leq \Isom(X)$ is its \emph{limit set} $\Lambda \subset \del X$, defined as the intersection of the closure of $G(\zero)$ with $\del X$. The limit set $\Lambda$ is both closed and $G$-invariant, and conversely
\begin{proposition}[{\cite[Th\'eor\`eme 5.1]{Coornaert}} or {\cite[Theorem 7.4.1]{DSU}}]
\label{propositionminimal}
Fix $G\leq\Isom(X)$. Then any closed $G$-invariant subset of $\del X$ containing at least two points contains $\Lambda$.
\end{proposition}

%\begin{proposition}[Cardinality of the limit set by classification]
%\label{propositioncardinalitylimitset}
%Fix $G\leq\Isom(X)$.
%\begin{itemize}
%\item[(i)] If $G$ is elliptic, then $\LambdaG = \emptyset$.
%\item[(ii)] If $G$ is parabolic or inward focal with global fixed point $\xi$, then $\LambdaG = \{\xi\}$.
%\item[(iii)] If $G$ is lineal with fixed pair $\{\xi_1,\xi_2\}$, then $\LambdaG \subset \{\xi_1,\xi_2\}$, with equality if $G$ is a group.
%\item[(iv)] If $G$ is outward focal or of general type, then $\#(\LambdaG) \geq \#(\R)$. Equality holds if $X$ is separable.
%\end{itemize}
%\end{proposition}
%
%\begin{definition}
%\label{definitionelementary}
%Fix $G\leq\Isom(X)$. $G$ is called \emph{elementary} if $\#(\Lambda) < \infty$ and \emph{nonelementary} if $\#(\Lambda) = \infty$.
%\end{definition}

%%%%%%%%%%%%%%%%%%%%%%%%%%%%%%%%%%
\bigskip
\subsection{The quasiconvex core}
\label{subsectionconvexhulls}
%%%%%%%%%%%%%%%%%%%%%%%%%%%%%%%%%%

\begin{definition}
\label{definitionconvexhull}
The \emph{quasiconvex hull} of a set $S\subset\bord X$ is the set
Given $S\subset\bord X$, let
\[
\Hull_1(S) := \bigcup_{x,y\in S} \geo xy.
\]
The \emph{quasiconvex core} of a group $G\leq\Isom(X)$ is the set
\[
\CC_\zero := X\cap \cl{\Hull_1(G(\zero))}.
\]
\end{definition}

Although the quasiconvex core depends on the distinguished point $\zero$, for any $x,y\in X$ the Hausdorff distance between the sets $\CC_x$ and $\CC_y$ is finite \cite[Proposition 7.5.9]{DSU}.

\begin{remark}
The quasiconvex hull of a set $S\subset\bord X$ is in general smaller than the \emph{convex hull}, which is by definition the smallest convex subset of $\bord X$ that contains $S$. However, for a general CAT(-1) space the operation of taking the convex hull may be quite badly behaved; cf. \cite[Remark 7.5.6]{DSU} for a more detailed discussion with references.% \cite[Corollary C]{Ancona} and \cite[Theorem 1]{Borbely}
\end{remark}

\subsection{Horoballs}
\label{subsectionhoroballs}

\begin{definition}
\label{definitionhoroball}
A \emph{horoball} is a set of the form
\[
H_{\xi,t} = \{x\in X : \busemann_\xi (\zero, x) > t\},
\]
where $\xi\in\del X$ and $t\in\R$. The point $\xi$ is called the \emph{center} of a horoball $H_{\xi,t}$, and will be denoted $\Center(H_{\xi,t})$. Note that for any horoball $H$, we have
\[
\cl H \cap \del X = \{\Center(H)\}.
\]
\end{definition}

%%%%%%%%%%%%%%%%%%%%%%%%%%%%%%%%
\subsection{Parabolic fixed points}
\label{sectionparabolic}
%%%%%%%%%%%%%%%%%%%%%%%%%%%%%%%%

\begin{definition}
\label{definitionparabolic}
A group $G\leq\Isom(X)$ is \emph{parabolic} if $G(\zero)$ is an unbounded set and $G$ has a global fixed point $\xi\in \Fix(G)$ such that 
\[
g'(\xi) = 1 \all g\in G,
\]
i.e. $\xi$ is neutral with respect to every element of $G$.
\end{definition}

\begin{definition}
\label{definitionLbp}
Let $G\leq\Isom(X)$. A point $\xi\in\del X$ is a \emph{parabolic fixed point} of $G$ if the semigroup
\[
G_\xi := \Stab(G;\xi) = \{g\in G:g(\xi) = \xi\}
\]
is a parabolic group. If in addition there exists a set $S\subset X$ whose closure does not contain $\xi$ such that
\[
G(\zero) \subset G_\xi(S),
\]
then $\xi$ is called a \emph{bounded parabolic point}. We denote the set of bounded parabolic poitns by $\Lbp(G)$.
\end{definition}

\begin{lemma}[{\cite[Lemma 12.3.6]{DSU}}]
\label{lemmaboundedparabolic}
Let $\xi$ be a bounded parabolic limit point of $G$, and let $H$ be a horoball centered at $\xi$ satisfying $G(\zero)\cap H = \emptyset$. Then there exists $\rho > 0$ such that
\begin{equation}
\label{LbpD}
\CC_\zero\cap\del H \subset G_\xi(B(\zero,\rho)).
\end{equation}
\end{lemma}

\subsection{Quasi-isometric embeddings of $G$ into $X$}

\begin{definition}
Let $X,Y$ be metric spaces. A map $f:X\to Y$ is called a \emph{quasi-isometric embedding} if for all $x_1,x_2\in X$,
\[
\dist_Y(f(x_1),f(x_2)) \asymp_{\plus,\times} \dist_X(x_1,x_2).
\]
A set $S\subset Y$ is \emph{cobounded} if there exists $\rho > 0$ such that for all $y_1\in Y$, there exists $y_2\in S$ such that $\dist_Y(y_1,y_2) \leq \rho$. A quasi-isometric embedding whose image is cobounded is called a \emph{quasi-isometry}.
\end{definition}

If $G\leq\Isom(X)$ is a geometrically finite group without parabolic points, then by the Milnor--Schwarz lemma \cite[Proposition I.8.19]{BridsonHaefliger}, $G$ is finitely generated, and for any Cayley graph of $G$, the orbit map $g\mapsto g(\zero)$ is a quasi-isometric embedding. If $G$ is geometrically finite with parabolic points, then in general neither of these things is true.\Footnote{For examples of infinitely generated strongly discrete parabolic groups, see \cite[Examples 11.2.18 and 11.2.20]{DSU}; these examples can be extended to nonelementary examples by taking a Schottky product with a lineal group. 
\cite[Theorem 11.2.6]{DSU} guarantees that the orbit map of a parabolic group is never a quasi-isometric embedding.} Nevertheless, by considering a certain \emph{weighted} Cayley metric with infinitely many generators, we can recover the rough metric structure of the orbit $G(\zero)$.

%\begin{definition}
%\label{definitiongraphmetrization}
%A \emph{weighted undirected graph} is a triple $(V,E,\ell)$, where $V$ is a nonempty set, $E\subset V\times V\butnot\{(x,x) : x\in V\}$ is invariant under the map $(x,y)\mapsto (y,x)$, and $\ell:E\to(0,\infty)$ is also invariant under $(x,y)\to (y,x)$. (If $\ell \equiv 1$, the graph is called \emph{unweighted}, and can be denoted simply $(V,E)$.) The \emph{path metric} on $V$ is the metric
%\begin{equation}
%\label{pathmetric}
%\dist_{E,\ell}(x,y) := \inf\left\{ \sum_{i = 0}^{n - 1} \ell(z_i,z_{i + 1}) : z_0 = x,\; z_n = y,\; (z_i,z_{i + 1})\in E \all i = 0,\ldots,n - 1\right\}.
%\end{equation}
%The \emph{geometric realization} of the graph $(V,E,\ell)$ is the metric space
%\[
%X = X(V,E,\ell) = \left(V\cup\bigcup_{(v,w)\in E} [0,\ell(v,w)]\right)/\sim,
%\]
%where $\sim$ represents the following identifications:
%\begin{align*}
%v &\sim ((v,w),0) \all (v,w)\in E\\
%((v,w),t) &\sim ((w,v),\ell(v,w) - t) \all (v,w)\in E \all t\in[0,\ell(v,w)]
%\end{align*}
%and the metric $\dist$ on $X$ is given by
%\[
%\dist\big(((v_0,v_1),t),((w_0,w_1),s)\big) = \min\{ |t - i\ell(v_0,v_1)| + \dist(v_i,w_j) + |s - j\ell(w_0,w_1)| : i = 0,1, j = 0,1\}.
%\]
%(The geometric realization of a graph is sometimes also called a graph. In the sequel, we shall call it a \emph{geometric graph}.)
%\end{definition}

\begin{definition}
Let $\Gamma$ be a group, let $E_0\subset \Gamma$ be a generating set, and let $\ell_0:E_0\to (0,\infty)$. Assume that for all $g\in E_0$, we have $g^{-1}\in E_0$ and $\ell_0(g^{-1}) = \ell_0(g)$. Then the \emph{weighted Cayley metric} on $\Gamma$ corresponding to the generating set $E_0$ and the weight function $\ell_0$ is the metric $\dist_\Gamma$ given by the formula
\[
\dist_\Gamma(g_1,g_2) := \inf_{\substack{(h_i)_1^n\in E_0^n \\ g_1 = g_2 h_1\cdots h_n}} \sum_{i = 1}^n \ell_0(h_i).
\]
\end{definition}

Let $G\leq\Isom(X)$ be a geometrically finite group. In what follows, we describe a generating set and a weight function whose weighted Cayley metric recovers the rough metric structure of $G(\zero)$. Let $P$ be a complete set of inequivalent parabolic points of $G$, and consider the set
\[
E := \bigcup_{\bp\in P} G_\bp,
\]
and for each $h\in G$ let
\begin{equation}
\label{l0def}
\ell_0(h) := 1\vee\dogo h.
\end{equation}
When $G$ is endowed with its weighted Cayley metric corresponding to the generating set $E\cup F$ (where $F$ is a sufficiently large finite set) and the weight function $\ell_0$, then the orbit map will be a quasi-isometric embedding:

\begin{theorem}[{\cite[Theorem 12.4.14]{DSU}}]
\label{theoremGF}
If $G\leq\Isom(X)$ is geometrically finite (cf. Definition \ref{definitionGF}), then
\begin{itemize}
\item[(i)] There exists a finite set $F$ such that $G$ is generated by $E\cup F$.
\item[(ii)] The orbit map $g\mapsto g(\zero)$ is a quasi-isometric embedding with respect to the weighted Cayley metric corresponding to the generating set $E\cup F$ and the weight function \eqref{l0def}.
\end{itemize}
\end{theorem}

%%%%%%%%%%%%%%%%%%%%%%%%%%%%%%%%%%
\draftnewpage
\section{Proof of Theorem \ref{theoremtukia}}
\label{sectiontukia}
%%%%%%%%%%%%%%%%%%%%%%%%%%%%%%%%%%

In this section, the notation and assumptions will be as in Theorem \ref{theoremtukia}.

It follows e.g from \cite[Theorem 6.2.3]{DSU} that a subgroup of $G$ is parabolic if and only if it is infinite and consists only of parabolic and elliptic elements. Since $\Phi$ is type-preserving, it follows that $\Phi$ preserves the class of parabolic subgroups, and also the class of maximal parabolic subgroups. But all maximal parabolic subgroups of $G$ are of the form $G_\xi$, where $\xi$ is a parabolic fixed point of $G$. It follows that there is a bijection $\phi:\Lbp(G)\to\Lbp(\w G)$ such that $\Phi(G_\xi) = \w G_{\phi(\xi)}$ for all $\xi\in\Lbp(G)$. The equivariance of $\phi$ implies that $\w P := \phi(P)$ is a complete set of inequivalent parabolic points for $\w G$.

Let $\dist_G$ and $\dist_{\w G}$ denote the weighted Cayley metrics on $G$ and $\w G$, respectively, with respect to the generating set and weight function of Theorem \ref{theoremGF}.

\begin{lemma}
\label{lemmatukia1}
$\dist_G \asymp_\times \dist_{\w G}\circ\Phi$.
\end{lemma}
\begin{proof}
Let $E$ and $F$ be as in Theorem \ref{theoremGF}, and let $\w E$ and $\w F$ be the corresponding sets for $\w G$. Since $\w P = \phi(P)$, we have $\w E = \Phi(E)$. On the other hand, for all $h\in E$, we have $\ell_0(h) \asymp_\times \ell_0(\Phi(h))$ by \eqref{tukia}. Thus, edges in the weighted Cayley graph of $G$ have roughly (multiplicatively asymptotically) the same weight as their corresponding edges in the weighted Cayley graph of $\w G$. (The sets $F$ and $\w F$ are both finite, and so their edges are essentially irrelevant.) The lemma follows.
\end{proof}

Thus, the map $\Phi(g(\zero)) := \Phi(g)(\zero)$ is a quasi-isometry between $G(\zero)$ and $\w G(\zero)$. At this point, we would like to extend $\Phi$ to an equivariant homeomorphism between $\Lambda$ and $\w\Lambda$. However, all known theorems that give such extensions, e.g. \cite[Theorem 6.5]{BonkSchramm}, require the spaces in question to be geodesic or at least roughly geodesic -- for the good reason that the extension theorems are false without this hypothesis\Footnote{A counterexample is given by letting
\begin{align*}
X_1 = X_2 = \mathbb R , && \dist_1(x,y) = \log(1 + |y - x|), && \dist_2(x,y) = \begin{cases} \dist_1(x,y) & xy\geq 0 \\ \dist_1(0,x) + \dist_1(0,y) & xy\leq 0 \end{cases},
\end{align*}
and letting $\Phi:X_1\to X_2$ be the identity map -- since $\#(\del X_1) = 1 < 2 = \#(\del X_2)$, $\Phi$ cannot be extended to a homeomorphism between $\del X_1$ and $\del X_2$. On the other hand, if one of the spaces in question is geodesic, then the extension theorem can be proven by isometrically embedding the other space into a geodesic hyperbolic metric space via \cite[Theorem 4.1]{BonkSchramm} -- a fact that however has no relevance to the present situation.} -- but the spaces $G(\zero)$ and $\w G(\zero)$ are not roughly geodesic. They are, however, embedded in the roughly geodesic metric spaces $\CC_\zero$ and $\w\CC_\zero$, which suggests the strategy of extending the map $\Phi$ to a quasi-isometry between $\CC_\zero$ and $\w\CC_\zero$. It turns out that this strategy works if we assume \eqref{tukia2}, and thus proves the existence of a \emph{quasisymmetric} equivariant homeomorphism between $\Lambda$ and $\w\Lambda$ in that case. Since we know that the equivariant homeomorphism is not necessarily quasisymmetric if \eqref{tukia2} fails (Example \ref{exampletukia} and Remark \ref{remarktukia}), this strategy can't be used to prove part (i) of Theorem \ref{theoremtukia}. Thus the proof splits into two parts at this point, depending on whether we have the stronger assumption \eqref{tukia2} that guarantees quasisymmetry, or only the weaker assumption \eqref{tukia}.

\subsection{Completion of the proof assuming \eqref{tukia2}}
The proof technique here is similar to \cite{Schwartz2}, as described to us by Marc Bourdon.

\begin{lemma}
\label{lemmatukia2}
Fix $\bp\in P$ and let $\w\bp = \phi(\bp)$. Let
\[
A = A(\bp) = \bigcup_{h\in G_\bp} \geo{h(\zero)}{\bp},
\]
and define a bijection $\psi = \psi_\bp:A\to \w A := A(\w\bp)$ by
\[
\psi(\geo{h(\zero)}{\bp}_t) = \geo{\Phi(h)(\zero)}{\w\bp}_{\alpha_\bp t}.
\]
Then $\psi$ is a quasi-isometry.
\end{lemma}
\begin{proof}
Fix two points $x_i = \geo{h_i(\zero)}{\bp}_{t_i} \in A$, $i = 1,2$. Write $y_i = h_i(\zero)$, $i = 1,2$. Then
%\[
%\lb y_1|\bp\rb_{y_2} = \lb y_2|\bp\rb_{y_1} = t := \frac12\dist(y_1,y_2).
%\]
%It follows that
\begin{equation}
\label{distinhoroball}
\dist(x_1,x_2) \asymp_\plus |t_2 - t_1| \vee (\dist(y_1,y_2) - t_1 - t_2).
\end{equation}
(This can be seen e.g. by repeated application of Proposition \ref{propositionrips}(ii).) On the other hand, if we write $\w y_i = \Phi(h_i)(\zero)$, $\w t_i = \alpha_\bp t_i$, and $\w x_i = \geo{\w y_i}{\w\bp}_{\w t_i}$, then by \eqref{tukia2} we have $\dist(\w y_1,\w y_2) \asymp_\plus \alpha_\bp \dist(y_1,y_2)$; applying \eqref{distinhoroball} along with its tilded version, we see that $\dist(\w x_1,\w x_2) \asymp_\plus \alpha_\bp \dist(x_1,x_2)$.
\end{proof}

For $g(\bp)\in G(P) = \Lbp(G)$, write $A_{g(\bp)} = g(A_\bp)$ and $\psi_{g(\bp)} = \Phi(g)\circ\psi_\bp\circ g^{-1}$; then $\psi_{g(\bp)}: A_{g(\bp)}\to A_{\phi(g(\bp))}$ is a quasi-isometry, and the implied constants are independent of $g(\bp)$. Let
\[
S = S(G) = \bigcup_{\xi\in \Lbp(G)} A_\xi \supset G(\zero),
\]
and define $\psi:S\to \w S := S(\w G)$ by letting
\[
\psi(x) = \psi_\xi(x) \all \xi\in \Lbp(G) \all x\in A_\xi.
\]
Note that for $g\in G$, $\psi(g(\zero)) = \Phi(g)(\zero)$.
\begin{lemma}
\label{lemmapsiquasi}
$\psi$ is a quasi-isometry.
\end{lemma}
\begin{proof}
Fix two points $x_1,x_2\in S$. For each $i = 1,2$, write $x_i \in A_{g_i(\bp_i)}$ for some $g_i(\bp_i)\in \Lbp(G_i)$. If $g_1(\bp_1) = g_2(\bp_2)$, then $\dist(\psi(x_1),\psi(x_2)) \asymp_\plus \dist(x_1,x_2)$ by Lemma \ref{lemmatukia2}. Otherwise, let $t > 0$ be large enough so that the collection $\scrH = \{H_{g(\bp)} := g(H_{\bp,t}) : g\in G, \; \bp\in P\}$ is disjoint. Then $y_i := \geo{x_i}{g_i(\bp_i)}_t \in H_{g_i(\bp_i)}$. It follows that the geodesic $\geo{y_1}{y_2}$ intersects both $\del H_{g_1(\bp_1)}$ and $\del H_{g_2(\bp_2)}$ (cf. Figure \ref{figurepsiquasi}), say in the points $z_1,z_2$. By Lemma \ref{lemmaboundedparabolic}, there exist points $w_i\in g_i G_{\bp_i}(\zero)$ such that $\dist(z_i,w_i) \asymp_\plus 0$. To summarize, we have
\[
\dist(x_1,x_2) \asymp_{\plus,t} \dist(y_1,y_2) = \dist(z_1,z_2) + \sum_{i = 1}^2 \dist(y_i,z_i) \asymp_{\plus,t} \dist(w_1,w_2) + \sum_{i = 1}^2 \dist(x_i,w_i).
\]
As $x_i,w_i\in A_{g_i(\bp_i)}$, we have $\dist(\psi(x_i),\psi(w_i)) \asymp_\plus \dist(x_i,w_i)$ by Lemma \ref{lemmatukia2}. On the other hand, since $w_1,w_2\in G(\zero)$, we have $\dist(\w w_1,\w w_2) \asymp_{\plus,\times} \dist(w_1,w_2)$ by Lemma \ref{lemmatukia1} and Theorem \ref{theoremGF}(ii). (Here $\w x = \psi(x)$.) Thus,
\[
\dist(x_1,x_2) \asymp_{\plus,\times} \dist(\w w_1,\w w_2) + \sum_{i = 1}^2 \dist(\w x_i,\w w_i) \geq \dist(\w x_1,\w x_2).
\]
Since the situation is symmetric, the reverse inequality holds as well.
\end{proof}

\begin{figure}
\begin{center} 
\begin{tikzpicture}[line cap=round,line join=round,>=triangle 45,x=1.0cm,y=1.0cm]
\clip(-3.674,-3.239) rectangle (3.698,3.586);
\draw(0.0,0.0) circle (3.0cm);
\draw(0.0,2.0) circle (1.0cm);
\draw(-1.6765383941122631,-0.4770413063924982) circle (1.2569145430600153cm);
\draw (-2.884984519391117,-0.8227176446835247)-- (-0.8537077377087111,-0.24806336779043003);
\draw (0.0,3.0)-- (0.0,0.5410285663267452);
\draw [shift={(-4.432995402589801,3.34299557285843)}] plot[domain=5.372236452220289:5.900827991552244,variable=\t]({1.0*4.778027883453108*cos(\t r)+-0.0*4.778027883453108*sin(\t r)},{0.0*4.778027883453108*cos(\t r)+1.0*4.778027883453108*sin(\t r)});
\begin{scriptsize}
\draw[color=black] (0.51,2.093) node {$H_{g_2(p_2)}$};
\draw [fill=black] (-2.884984519391117,-0.8227176446835247) circle (1pt);
\draw[color=black] (-3.3,-0.887) node {$g_1(p_1)$};
\draw[color=black] (-2.0,0.305) node {$H_{g_1(p_1)}$};
\draw [fill=black] (0.0,3.0) circle (1pt);
\draw[color=black] (-0.013254332234670846,3.2) node {$g_2(p_2)$};
\draw [fill=black] (-0.8537077377087111,-0.24806336779043003) circle (1pt);
\draw[color=black] (-0.8250955547283336,-0.45) node {$x_1$};
\draw [fill=black] (0.0,0.5410285663267452) circle (1pt);
\draw[color=black] (0.30213211454936214,0.5204766261395189) node {$x_2$};
\draw [fill=black] (-1.5040817556787958,-0.43205612268083204) circle (1pt);
\draw[color=black] (-1.5209594597229017,-0.65) node {$y_1$};
\draw [fill=black] (0.0,1.56027231456143) circle (1pt);
\draw[color=black] (0.2855639263352058,1.5974088600596753) node {$y_2$};
\end{scriptsize}
\end{tikzpicture}
\caption{The proof of Lemma \ref{lemmapsiquasi}. The distance between $y_1$ and $y_2$ is broken up into three segments, each of which is coarsely asymptotically preserved upon applying $\psi$.}
\label{figurepsiquasi}
\end{center}
\end{figure}
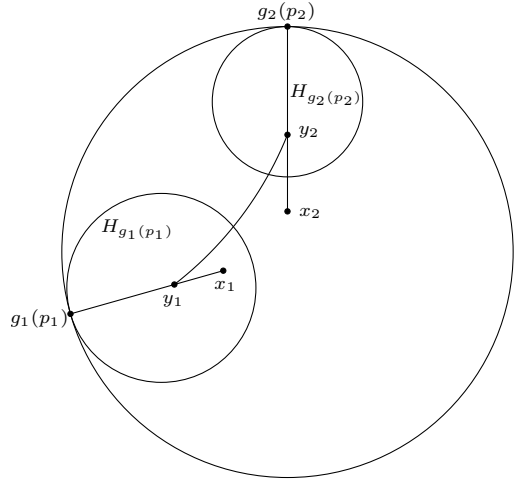

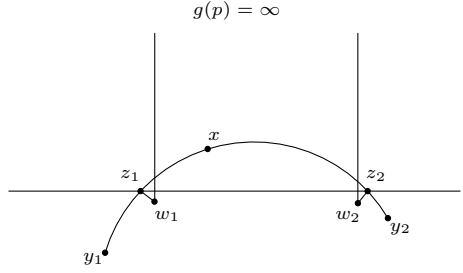
\begin{figure}
\begin{center} 
\begin{tikzpicture}[line cap=round,line join=round,>=triangle 45,x=1.0cm,y=1.0cm]
\clip(-3.331733658928572,0.17471897426577723) rectangle (3.5450834715133723,4.0257365673132615);
\draw (-3.0,1.5)-- (3.0,1.5);
\draw [shift={(0.26067583481434903,0.08708448064792945)}] plot[domain=0.5343314243396922:2.8507321995474353,variable=\t]({1.0*2.0675167655199447*cos(\t r)+-0.0*2.0675167655199447*sin(\t r)},{0.0*2.0675167655199447*cos(\t r)+1.0*2.0675167655199447*sin(\t r)});
\draw (-1.06,1.36)-- (-1.06,3.6);
\draw (1.64,1.34)-- (1.64,3.6);
\draw (-1.2487264677147425,1.5)-- (-1.0637012613578571,1.3605243994346345);
\draw (1.641128303962467,1.3391423475348692)-- (1.77007813734344,1.5);
\begin{scriptsize}
\draw [fill=black] (2.04,1.14) circle (1pt);
\draw[color=black] (2.2055990913229415,1.0) node {$y_2$};
\draw [fill=black] (-1.72,0.68) circle (1pt);
\draw[color=black] (-1.8726524590782814,0.5933078430752864) node {$y_1$};
\draw [fill=black] (-1.2487264677147425,1.5) circle (1pt);
\draw[color=black] (-1.3823054984728558,1.7055582659119821) node {$z_1$};
\draw [fill=black] (1.77007813734344,1.5) circle (1pt);
\draw[color=black] (1.8826876782413198,1.7055582659119821) node {$z_2$};
\draw [fill=black] (-1.0637012613578571,1.3605243994346345) circle (1pt);
\draw[color=black] (-0.891,1.17) node {$w_1$};
\draw [fill=black] (1.641128303962467,1.3391423475348692) circle (1pt);
\draw[color=black] (1.511,1.16) node {$w_2$};
%\draw [fill=black] (0.0,3.6) circle (1pt);
\draw[color=black] (0.03,3.9) node {$g(p) = \infty$};
\draw [fill=black] (-0.3562789079803063,2.0604046323337206) circle (1pt);
\draw[color=black] (-0.27005507563615844,2.2317842724153647) node {$x$};
\end{scriptsize}
\end{tikzpicture}
\caption{The proof of Lemma \ref{lemmaSCB}, in the upper half-space model (cf. e.g. \cite[\62.5.2]{DSU}). The thin triangles condition guarantees that $x$ is close to one of the geodesics $\geo{w_1}{g(p)}$, $\geo{w_2}{g(p)}$, both of which are contained in $S$.}
\label{figureSCB}
\end{center}
\end{figure}

\begin{lemma}
\label{lemmaSCB}
$S$ is cobounded in $C = \CC_\zero$.
\end{lemma}
\begin{proof}
Fix $x\in\CC_\zero$. If $x\notin \bigcup\scrH$, then $\dist(x,S)\leq \dist(x,G(\zero))\asymp_\plus 0$. So suppose $x\in H = H_{g(\bp)}$ for some $g\in G$, $\bp\in P$. Write $x\in \geo{y_1}{y_2}$ for some $y_1,y_2\in G(\zero)$. Then there exist $z_1,z_2\in \geo{y_1}{y_2}\cap \del H$ such that $x\in \geo{z_1}{z_2}$. By Lemma \ref{lemmaboundedparabolic}, there exist $w_1,w_2\in g G_\bp(\zero)$ such that $\dist(z_i,w_i)\asymp_\plus 0$. It follows that $\lb w_1|w_2\rb_x\asymp_\plus 0$. By Proposition \ref{propositionrips}, we have
\[
\dist(x,S) \leq \dist(x,S_{g(\bp)}) \leq \dist(x,\geo{w_1}{g(\bp)}\cup\geo{w_2}{g(\bp)}) \asymp_\plus 0
\]
(cf. Figure \ref{figureSCB}). This completes the proof.
\end{proof}

Thus, the embedding map from $S$ to $C$ is an equivariant quasi-isometry. Thus $S$, $C$, $\w S$, and $\w C$ are all equivariantly quasi-isometric. By \cite[Theorem 6.5]{BonkSchramm}, the quasi-isometry between $C$ and $\w C$ extends to a quasisymmetric homeomorphism between $\del C = \Lambda$ and $\del\w C = \w\Lambda$. This completes the proof of Theorem \ref{theoremtukia}(ii).

%Let $\Psi:C\to \w C$ be an equivariant quasi-isometry.
%
%\begin{lemma}
%For all $x,y,z\in C$,
%\[
%\lb \Psi(x)|\Psi(y)\rb_{\Psi(z)} \asymp_{\plus,\times} \lb x|y\rb_z.
%\]
%\end{lemma}
%\begin{proof}
%By the Morse Lemma (e.g. \ref[Theorem 9.38]{DrutuKapovich}), there exists $C > 0$ such that for all $x,y\in C$, the Hausdorff distance between $\Psi(\geo xy)$ and $\geo{\Psi(x)}{\Psi(y)}$ is less than $C$. Combining with Proposition \ref{propositionrips}(i), we see that
%\begin{align*}
%\lb \w x|\w y\rb_{\w z} &\asymp_\plus \dist(\w z,\geo{\w x}{\w y})
%&\asymp_\plus \dist(\w z,\Psi(\geo xy))
%\asymp_{\plus,\times} \dist(z,\geo xy) \asymp_\plus \lb x|y\rb_z.
%\end{align*}
%\end{proof}
%
%It follows that $\Psi$ sends Gromov sequences to Gromov sequences, and thus it extends uniquely to a map $\del\Psi:\del C = \Lambda\to\del\w C = \w\Lambda$ which is continuous on $\Lambda$. To show that $\del\Psi$ is quasisymmetric, take $\bp,\eta_1,\eta_2\in \Lambda$, 

\subsection{Completion of the proof assuming only \eqref{tukia}}
We begin by recalling the Morse lemma:

\begin{definition}
\label{definitionquasigeodesic}
A path $\gamma:[a,b]\to X$ is a \emph{$K$-quasigeodesic} if for all $a\leq t_1 < t_2 \leq b$,
\[
\frac1K (t_2 - t_1) - K \leq \dist(\gamma(t_1),\gamma(t_2)) \leq K(t_2 - t_1) + K.
\]
(In other words, $\gamma$ is a $K$-quasigeodesic if $\dist(\gamma(t_1),\gamma(t_2)) \asymp_{\plus,\times} t_2 - t_1$, and the implied constants are both equal to $K$.)
\end{definition}

\begin{lemma}[Morse Lemma, {\cite[Theorem 9.38]{DrutuKapovich}}]
\label{lemmamorse}
For every $K > 0$, there exists $K_2 > 0$ such that the Hausdorff distance between any $K$-quasigeodesic $\gamma$ and the geodesic $\geo{\gamma(a)}{\gamma(b)}$ is at most $K_2$.
\end{lemma}

\begin{lemma}
\label{lemmaquasi}
Fix $h_1,\ldots,h_n\in E\cup F$, let $g_k = h_1\cdots h_k$ and $x_k = g_k(\zero)$ for all $k = 0,\ldots,n$, and suppose that
\begin{equation}
\label{quasi}
\dist(x_k,x_\ell) \asymp_\times \sum_{i = k + 1}^\ell \ell_0(h_i) \all 0\leq k < \ell \leq n.
\end{equation}
Then the path $\gamma = \bigcup_{k = 0}^{n - 1} \geo{x_k}{x_{k + 1}}$ is a $K$-quasigeodesic, where $K > 0$ is independent of $h_1,\ldots,h_n$.
\end{lemma}
\begin{proof}
Fix $0 <  k \leq \ell < n$ and points $z\in\geo{x_{k - 1}}{x_k}$, $w\in\geo{x_\ell}{x_{\ell + 1}}$. To show that $\gamma$ is a quasigeodesic, it suffices to show that
\begin{equation}
\label{ETSquasi}
\dist(z,w) \gtrsim_{\plus,\times} \dist(z,x_k) + \dist(x_k,x_\ell) + \dist(x_\ell,w).
\end{equation}
\begin{claim}
\label{claimquasi}
$\dist(z,w)\gtrsim_\plus \min(\dist(z,x_{k - 1}),\dist(z,x_k))$.
\end{claim}
\begin{subproof}
If $h_k\in F$, then $\dist(z,x_k) \leq \dist(x_{k - 1},x_k) \asymp_\plus 0$, so $\dist(z,w)\gtrsim_\plus \dist(z,x_k)$. Thus, suppose that $h_k\in E$; then $h_k\in G_\bp$ for some $\bp\in P$. Let $g = g_{k - 1}$; since $g^{-1}(z) \in \geo\zero{h_k(\zero)}$, by Proposition \ref{propositionrips}(i) we have $\dist(g^{-1}(z),\geo y\bp) \asymp_\plus 0$, where either $y = \zero$ or $y = h_k(\zero)$.
\begin{subclaim}
\label{subclaimquasi}
There exists $t > 0$ independent of $h_1,\ldots,h_n$ such that $g^{-1}(w)\notin H_{\bp,t}$.
\end{subclaim}
(Cf. Figure \ref{figurequasi}.)

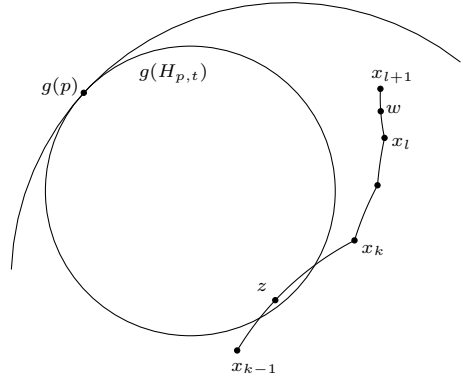
\begin{figure}
\begin{center} 
\begin{tikzpicture}[line cap=round,line join=round,>=triangle 45,x=1.0cm,y=1.0cm]
\clip(-3.8491074720597833,-2.1266876752711386) rectangle (3.5797459638630347,3.7907291690819807);
\draw [shift={(0.0,-0.0)}] plot[domain=0.9045225827260511:3.1035675129286187,variable=\t]({1.0*3.6889564920177627*cos(\t r)+-0.0*3.6889564920177627*sin(\t r)},{0.0*3.6889564920177627*cos(\t r)+1.0*3.6889564920177627*sin(\t r)});
\draw(-1.3068803163851197,1.1823748155886806) circle (1.9266288587793043cm);
\draw [shift={(2.685346326431731,-2.9662333458959456)}] plot[domain=2.049120890160455:2.6001311772486844,variable=\t]({1.0*3.9319392309623806*cos(\t r)+-0.0*3.9319392309623806*sin(\t r)},{0.0*3.9319392309623806*cos(\t r)+1.0*3.9319392309623806*sin(\t r)});
\draw [shift={(4.767483130446365,-0.6701022273279595)}] plot[domain=2.648334556608196:2.8438017124837423,variable=\t]({1.0*4.071162818980422*cos(\t r)+-0.0*4.071162818980422*sin(\t r)},{0.0*4.071162818980422*cos(\t r)+1.0*4.071162818980422*sin(\t r)});
\draw [shift={(5.978973518661347,0.8697896686860462)}] plot[domain=2.9284471092129145:3.0609325657134683,variable=\t]({1.0*4.812997880472604*cos(\t r)+-0.0*4.812997880472604*sin(\t r)},{0.0*4.812997880472604*cos(\t r)+1.0*4.812997880472604*sin(\t r)});
\draw [shift={(3.7794856280301645,2.426664585692047)}] plot[domain=3.0967788621557717:3.353471214637198,variable=\t]({1.0*2.5618834704300895*cos(\t r)+-0.0*2.5618834704300895*sin(\t r)},{0.0*2.5618834704300895*cos(\t r)+1.0*2.5618834704300895*sin(\t r)});
\begin{scriptsize}
\draw [fill=black] (-2.723168609785217,2.4885242057654264) circle (1pt);
\draw[color=black] (-3.036353833196826,2.5787281286723056) node {$g(p)$};
\draw[color=black] (-1.53,2.75) node {$g(H_{p,t})$};
\draw [fill=black] (-0.6841540363190708,-0.9397552627666829) circle (1pt);
\draw[color=black] (-0.45369,-1.18728) node {$x_{k-1}$};
\draw [fill=black] (0.8755044855278509,0.524413961824313) circle (1pt);
\draw[color=black] (1.127,0.397) node {$x_k$};
\draw [fill=black] (1.181623953260003,1.257585677919404) circle (1pt);
%\draw[color=black] (1.355367583817047,1.1528445517197468) node {$I$};
\draw [fill=black] (1.2748920600343736,1.887908634247394) circle (1pt);
\draw[color=black] (1.4994382699732513,1.81) node {$x_{l}$};
\draw [fill=black] (1.2242241363497024,2.2425841000400935) circle (1pt);
\draw[color=black] (1.398,2.27) node {$w$};
\draw [fill=black] (1.2201742115269398,2.541433873479229) circle (1pt);
\draw[color=black] (1.38,2.7) node {$x_{l+1}$};
\draw [fill=black] (-0.17745538908965663,-0.2709573946084274) circle (1pt);
\draw[color=black] (-0.3556927085260205,-0.10193299599850501) node {$z$};
\end{scriptsize}
\end{tikzpicture}
\caption{In Subclaim \ref{subclaimquasi}, the geodesics $\geo{x_{k - 1}}{x_k}$ and $\geo{x_\ell}{x_{\ell + 1}}$ cannot penetrate the same cusp, thus guaranteeing some distance between $z$ and $w$.}
\label{figurequasi}
\end{center}
\end{figure}
\begin{subproof}
If $h_{\ell + 1}\in F$, then $\dist(g^{-1}(w),g^{-1}(x_{\ell + 1})) \asymp_\plus 0$, in which case the subclaim follows from the fact that $\bp$ is a bounded parabolic point. Thus suppose $h_{\ell + 1}\in E$; then $h_{\ell + 1}\in G_\eta$ for some $\eta\in P$. Let $k = g_\ell$; since $k^{-1}(w)\in \geo\zero{h_{\ell + 1}(\zero)}$, by Proposition \ref{propositionrips}(i) we have $\dist(k^{-1}(w),\geo p\eta) \asymp_\plus 0$, where either $p = \zero$ or $p = h_k(\zero)$. In particular $\busemann_\eta(\zero,k^{-1}(w))\gtrsim_\plus 0$, so by the disjointness of the family $\scrH$, there exists $t > 0$ such that $k^{-1}(w)\notin j(H_{\bq,t})$ for all $\bq\in P$ and $j\in G$ such that $j(\bq)\neq \eta$. In particular, letting $j = k^{-1} g = (h_k\cdots h_\ell)^{-1}$ and $\bq = \bp$, we have $g^{-1}(w)\notin H_{\bp,t}$ unless $j(\bp) = \eta$. But if $j(\bp) = \eta$, then $j = \id$ due to the minimality $P$, and this contradicts \eqref{quasi}.
\end{subproof}
It follows that
\begin{align*}
\dist(z,w) \geq \busemann_\bp(g^{-1}(w),g^{-1}(z))
&= \busemann_\bp(\zero,g^{-1}(z)) - \busemann_\bp(\zero,g^{-1}(w))\\
&= \dist(y,g^{-1}(z)) - \busemann_\bp(\zero,g^{-1}(w)) \gtrsim_\plus \dist(y,g^{-1}(z)) - t.
\end{align*}
Applying $g$ to both sides finishes the proof of Claim \ref{claimquasi}.
\end{subproof}

A similar argument shows that $\dist(z,w)\gtrsim_\plus \min(\dist(w,x_\ell),\dist(w,x_{\ell + 1}))$. Now let $y_1\in \{x_{k - 1},x_k\}$ and $y_2\in \{x_\ell,x_{\ell + 1}\}$ be such that
\begin{equation}
\label{zwzyzwwy}
\dist(z,w) \gtrsim_\plus \dist(z,y_1) \text{ and } \dist(z,w) \gtrsim_\plus \dist(w,y_2).
\end{equation}
Then the triangle inequality gives $\dist(z,w)\gtrsim_{\plus,\times} \dist(y_1,y_2)$. On the other hand, \eqref{quasi} implies that $\dist(y_1,y_2) \gtrsim_{\plus,\times} \dist(y_1,x_k) + \dist(x_k,x_\ell) + \dist(x_\ell,y_2)$. Combining with \eqref{zwzyzwwy} and using the triangle inequality gives \eqref{ETSquasi}.
\end{proof}

\begin{lemma}
\label{lemmagromovquasi}
For all $x,y,z\in G(\zero)$,
\[
\lb \w x | \w y\rb_{\w z} \asymp_{\plus,\times} \lb x | y \rb_z.
\]
\end{lemma}
\begin{proof}
Fix $g_1,g_2\in G$, and we will show that
\begin{equation}
\label{ETSgromovquasi}
\lb \w g_1(\zero) | \w g_2(\zero)\rb_\zero \lesssim_{\plus,\times} \lb g_1(\zero) | g_2(\zero)\rb_\zero.
\end{equation}
The reverse inequality will then follow by symmetry. By Theorem \ref{theoremGF}(ii), there exists a sequence $h_1,\ldots,h_n\in E\cup F$ such that $g_2 = g_1 h_1\cdots h_n$ and satisfying \eqref{quasi}. By Lemma \ref{lemmatukia1}, the sequence $\w h_1,\ldots,\w h_n\in \w E\cup \w F$ also satisfies \eqref{quasi}. Let $x_k = g_1 h_1 \cdots h_k(\zero)$. By Lemma \ref{lemmaquasi}, the paths
\begin{align*}
\gamma &= \bigcup_{k = 0}^{n - 1} \geo{x_k}{x_{k + 1}}\\
\w\gamma &= \bigcup_{k = 0}^{n - 1} \geo{\w x_k}{\w x_{k + 1}}
\end{align*}
are quasigeodesics. So by Lemma \ref{lemmamorse}, $\gamma$ and $\w\gamma$ lie within a bounded Hausdorff distance of the geodesics they represent, namely $\geo{x_0}{x_n}$ and $\geo{\w x_0}{\w x_n}$. Combining with Proposition \ref{propositionrips}(i), we have
\[
\lb g_1(\zero) | g_2(\zero)\rb_\zero = \lb x_0 | x_n \rb_\zero \asymp_\plus \dist(\zero,\geo{x_0}{x_n}) \asymp_\plus \dist(\zero,\gamma),
\]
and similarly for $\w\gamma$. So to prove \eqref{ETSgromovquasi}, we need to show that $\dist(\zero,\w\gamma) \lesssim_{\plus,\times} \dist(\zero,\gamma)$.

Fix $z\in\gamma$, and we will show that $\dox z \gtrsim_{\plus,\times} \dist(\zero,\w\gamma)$. Write $z\in \geo{x_{k - 1}}{x_k}$ for some $k = 1,\ldots,n$. By Proposition \ref{propositionrips}(i), we have
\begin{align*}
\dist(\zero,\w\gamma) &\leq \dist(\zero,\geo{\w x_{k - 1}}{\w x_k}) \asymp_\plus \lb \w x_{k - 1} | \w x_k \rb_\zero\\
\dox z &\geq \dist(\zero,\geo{x_{k - 1}}{x_k}) \asymp_\plus \lb x_{k - 1} | x_k \rb_\zero,
\end{align*}
so to complete the proof of Lemma \ref{lemmagromovquasi} it suffices to show that
\begin{equation}
\label{gromovquasi}
\lb \w x_{k - 1} | \w x_k \rb_\zero \lesssim_{\plus,\times} \lb x_{k - 1} | x_k \rb_\zero.
\end{equation}
Now, if $h_k\in F$, then $\dist(x_{k - 1},x_k) \asymp_\plus \dist(\w x_{k - 1},\w x_k) \asymp_\plus 0$, so \eqref{gromovquasi} follows from Theorem \ref{theoremGF}(ii). Thus, suppose that $h_k\in E$, and write $h_k\in G_\bp$ for some $\bp\in P$. Use the notations $g = g_1 h_1\cdots h_{k - 1}$ and $h = h_k$, so that $x_{k - 1} = g(\zero)$ and $x_k = gh(\zero)$. Then for $y = \zero, h(\zero)$ we have (see \cite[Corollary 3.4.12]{DSU}) that
\[
\lb y|\bp\rb_{g^{-1}(\zero)} \asymp_\plus \frac12[\dist(g^{-1}(\zero),y) - \busemann_\bp(g^{-1}(\zero),y)] \gtrsim_\plus \frac12\dist(g^{-1}(\zero),y) \asymp_\times \dox{g(y)},
\]
so by Gromov's inequality,
\[
\lb x_{k - 1} | x_k \rb_\zero = \lb g(\zero)|gh(\zero)\rb_\zero = \lb \zero|h(\zero)\rb_{g^{-1}(\zero)} \gtrsim_{\plus,\times} \dox{g(y)} \asymp_{\plus,\times} \dox{\w g(\w y)} \geq \lb \w x_{k - 1} | \w x_k \rb_\zero. 
\]
This demonstrates \eqref{gromovquasi} and completes the proof of Lemma \ref{lemmagromovquasi}.
\end{proof}

It follows that the map $\Phi$ sends Gromov sequences to Gromov sequences, so it induces an equivariant homeomorphism $\del\Phi:\Lambda\to \w\Lambda$. This completes the proof of Theorem \ref{theoremtukia}(i).

%%%%%%%%%%%%%%%%%%%%%%%%%%%%%%%%%%
\draftnewpage
\section{Applications to finite-dimensional ROSSONCTs}
\label{sectiontukiaROSSONCT}
%%%%%%%%%%%%%%%%%%%%%%%%%%%%%%%%%%

A particularly interesting case of Theorem \ref{theoremtukia} is when $X$ and $\w X$ are both finite-dimensional ROSSONCTs. In this case, \eqref{tukia} always holds, but \eqref{tukia2} does not; nevertheless, there is a reasonable sufficient condition for \eqref{tukia2} to hold. Specifically, we have the following:

\begin{proposition}
\label{propositionCROSSONCTquasi}
Let $X$ and $\w X$ be finite-dimensional ROSSONCTs, let $G\leq\Isom(X)$ and $\w G\leq \Isom(\w X)$ be geometrically finite groups, and let $\Phi:G\to\w G$ be a type-preserving isomorphism. Fix $\bp\in P$, and let $\w\bp = \phi(\bp)\in \Lbp(\w G)$ be the unique point such that $\Phi(G_\bp) = \w G_{\w\bp}$. Then
\begin{itemize}
\item[(i)] \eqref{tukia} holds.
\item[(ii)] Let $H\leq G_\bp$ be a nilpotent subgroup of finite index. If the underlying base fields of $X$ and $\w X$ are the same, say $\F$, and if $\rank([H,H]) = \dim_\R(\F) - 1$, then \eqref{tukia2} holds.
\end{itemize}
\end{proposition}

Before we begin the proof of Proposition \ref{propositionCROSSONCTquasi}, it will be necessary to understand the structure of a parabolic subgroup of $\Isom(X)$.

Let $X = \H = \H_\F^d$ be a finite-dimensional ROSSONCT, let $\bp = [(1,1,\0)]$, and let $J_\bp = \Stab(\Isom(X);\bp)$. Note that $J_\bp$ is a parabolic group in the sense of Lie theory, while it is a focal group according to the classification of \cite[\66]{DSU} (and in particular not parabolic in the sense of Definition \ref{definitionparabolic}). To study the group $J_\bp$, we use the coordinate system generated by the basis
\[
\ff_0 = (\ee_0 + \ee_1)/2, \;\; \ff_1 = \ee_1 - \ee_0, \;\; \ff_i = \ee_i \;\;\; (i = 2,\ldots,d).
\]
In this coordinate system, the sesquilinear form $B_\QQ$ takes the form
\[
B_\QQ(\xx,\yy) = \wbar x_0 y_1 + \wbar x_1 y_0 + \sum_{i = 2}^d \wbar x_i y_i,
\]
the point $\bp$ takes the form $\bp = [\ff_0]$, and the group $J_\bp$ can be written (cf. \cite[Theorem 2.3.3]{DSU}) as
\[
J_\bp = \left\{ h_{\lambda,a,\vv,\ww,m,\sigma} := \left[\begin{array}{lll}
\lambda & a & \ww^\dag\\
& \lambda^{-1} &\\
& \vv & m
\end{array}\right] \sigma^{d + 1} :
\begin{split}
\lambda > 0,\; a\in \F,\;  \vv,\ww\in \F^{d - 1},\\
m\in \SO(\F^{d - 1};\EE),\; \sigma\in\Aut(\F)
\end{split}\right\} \cap \Isom(X),
\]
where $\EE$ denotes the Euclidean quadratic form on $\F^{d - 1}$. Given $\lambda,a,\vv,\ww,m$, it is readily verified that $h_{\lambda,a,\vv,\ww,m}\in \Isom(X)$ if and only if
\[
2\lambda^{-1}\Re(a) + \|\vv\|^2 = 0 \text{ and } \lambda^{-1}\ww^\dag + \vv^\dag m = \0.
\]
Consequently, it makes sense to rewrite $J_\bp$ as
\[
J_\bp = \left\{ h_{\lambda,a,\vv,m,\sigma} := \left[\begin{array}{lll}
\lambda & a - \lambda\|\vv\|^2/2 & -\lambda\vv^\dag m\\
& \lambda^{-1} &\\
& \vv & m
\end{array}\right] \sigma^{d + 1} :
\begin{split}
\lambda > 0,\; a\in \Im(\F),\;  \vv\in \F^{d - 1},\\
m\in \SO(\F^{d - 1};\EE),\; \sigma\in\Aut(\F)
\end{split}\right\}.
\]
We can now define the \emph{Langlands decomposition} of $J_\bp$:
\begin{align*}
M_\bp &= \{h_{1,0,\0,m,\sigma} : m\in\SO(\F^{d - 1};\EE),\; \sigma\in \Aut(\F)\}\\
A_\bp &= \{h_{\lambda,0,\0,I_{d - 1},e} : \lambda > 0\}\\
N_\bp &= \{n(a,\vv) := h_{1,a,\vv,I_{d - 1},e} : a\in\Im(\F) ,\; \vv\in \F^{d - 1}\}\\
J_\bp &= M_\bp A_\bp N_\bp.
\end{align*}
We observe the following facts about the Langlands decomposition: the groups $M_\bp$ and $A_\bp$ commute with each other and normalize $N_\bp$, which is nilpotent of order at most $2$. Moreover, the subgroup $M_\bp N_\bp$ is exactly the kernel of the homomorphism $J_\bp\ni h\mapsto h'(\bp)$, where $h'$ denotes the metric derivative. Equivalently, $M_\bp N_\bp$ is the largest parabolic subgroup of $J_\bp$, where ``parabolic'' is interpreted in the sense of Definition \ref{definitionparabolic}.

Let's look a bit more closely at the internal structure of $N_\bp$. The composition law is given by
\begin{equation}
\label{compositionlaw}
n(a_1,\vv_1) n(a_2,\vv_2) = n(a_1 + a_2 + \Im B_\EE(\vv_2,\vv_1),\vv_1 + \vv_2),
\end{equation}
confirming that $N_\bp$ is nilpotent of order at most two, and that its commutator is given by
\[
Z_\bp = \{n(a,\0) : a\in\Im(\F)\}.
\]
Moreover, the map $\pi:n(a,\vv)\mapsto \vv\in \F^{d - 1}$ is a homomorphism whose kernel is $Z_\bp$.

Now let $H\leq M_\bp N_\bp$ be a discrete parabolic subgroup. By Margulis's lemma, $H$ is almost nilpotent, and so by \cite[Lemma 3.4]{CorletteIozzi}, there exist a finite index subgroup $H_2\subset H$ and a homomorphism $\psi:H_2\to N_\bp$ such that $\psi(h)(\zero) = h(\zero)$ for all $h\in H_2$. (Here $\zero = [\ee_0] = [2\ff_0 - \ff_1]$ as usual.) We then let $H_3 = \psi(H_2) \leq N_\bp$.
\begin{definition}
\label{definitionHregular}
The group $H$ is \emph{regular} if $\pi(H_3)$ is a discrete subgroup of $\F^{d - 1}$. If $H$ is regular, we define its \emph{quasi-commutator} to be the subgroup
\[
Z = Z(H) = \psi^{-1}(Z_\bp) = \Ker(\pi\circ\psi) \leq H.
\]
Note that in general, the quasi-commutator of $H$ cannot be determined from its algebraic structure; cf. Example \ref{exampletukia}. Nevertheless, since $\F^{d - 1}$ is abelian, the quasi-commutator of $H$ always contains the commutator of $H_2$.
\end{definition}

In general, if $H\leq\Isom(X)$ is a discrete parabolic subgroup, we can conjugate the fixed point of $H$ to $[(1,1,\0)]$, apply the above construction, and then conjugate back to get a subgroup $Z(H)\leq H$.

If $H$ is regular, then the quasi-commutator $Z\leq H$ can be used to give an algebraic description of the function $h\mapsto \dogo h$. Specifically, we have the following:

\begin{lemma}
\label{lemmaCROSSONCTquasi}
Let $\dist_H$ and $\dist_Z$ be any Cayley metrics on $H$ and $Z$, respectively.
\begin{itemize}
\item[(i)]
\begin{equation}
\label{CROSSONCTquasi1}
\dogo h \asymp_{\plus,\times} 0\vee \log\dist_H(e,h).
\end{equation}
\item[(ii)] If $H$ is regular, then
\begin{equation}
\label{CROSSONCTquasi2}
\dogo h \asymp_\plus \min_{z\in Z} \big(0\vee 2\log\dist_H(z,h) \vee \log\dist_Z(e,z)\big)  \all h\in H.
\end{equation}
\end{itemize}
\end{lemma}
\begin{proof}
Let $F\subset H$ be a finite set so that $H_2 F = H$, and let $H_3 = \psi(H_2)$. Then for all $h\in H$, we can write $h = h_2 f$ for some $h_2\in H_2$ and $f\in F$, and then
\begin{align*}
\dogo h &\asymp_\plus \dogo{h_2} = \dogo{\psi(h_2)}\\
\dist_H(z,h) &\asymp_\plus \dist_H(z,h_2) \asymp_\times \dist_{H_2}(z,h_2) = \dist_{H_3}(\psi(z),\psi(h_2))\\
\min_{z\in Z} \big(0\vee 2\log\dist_H(z,h) \vee \log\dist_Z(e,z)\big) & \asymp_\plus \min_{z\in \psi(Z)} \big(0\vee 2\log\dist_{H_3}(z,\psi(h_2)) \vee \log\dist_{\psi(Z)}(e,z)\big).
\end{align*}
Thus, we may without loss of generality assume that $H = H_3$, i.e. that $H\leq N_\bp$ and $Z_H = H\cap Z_\bp$. We can also without loss of generality assume that $\bp = [(1,1,\0)]$.

The following formula regarding the function $n(a,\vv)$ can be verified by direct computation (cf. \cite[(3.5)]{CorletteIozzi}):
\begin{equation}
\label{corletteiozzi}
\|n(a,\vv)\| \asymp_\plus 0\vee 2\log\|\vv\| \vee \log|a|
\end{equation}
On the other hand, iterating \eqref{compositionlaw} gives
\begin{equation}
\label{compositionlawbounds}
\begin{split}
\|\vv\| &\lesssim_\times \dist_H(e,n(a,\vv))\\
|a| &\lesssim_\times \dist_H(e,n(a,\vv))^2\\
|a| &\lesssim_\times \dist_Z(e,n(a,\0)).
\end{split}
\end{equation}
These formulas make it easy to verify the $\lesssim$ direction of \eqref{CROSSONCTquasi2}: given $h = n(a,\vv)\in H$ and $z = n(b,\0)\in Z$, we have
\begin{align*}
0\vee 2\log\dist_H(z,h) \vee \log\dist_Z(e,z)
&=_\pt 0\vee 2\log\dist_H(e,n(a - b,\vv)) \vee \log\dist_Z(e,n(b,\0))\\
&\geq_\pt 0\vee 2\log\big(\|\vv\|\vee \sqrt{|a - b|}\big)\vee \log|b|\\
&=_\pt 0\vee 2\log\|\vv\| \vee \log|a - b| \vee\log|b|\\
&\gtrsim_\plus 0\vee 2\log\|\vv\| \vee \log|a| \asymp_\plus \dogo{n(a,\vv)} = \dogo h.
\end{align*}
Setting $z = e$ yields the $\lesssim$ direction of \eqref{CROSSONCTquasi1}.

To prove the $\gtrsim$ directions, we will need the following easily verified fact:

\begin{fact}
\label{factFDVS}
If $V$ is a finite-dimensional vector space, $\Lambda\leq V$ is a discrete subgroup, and $\dist_\Lambda$ is a Cayley metric on $\Lambda$, then $\dist_\Lambda(\0,\vv) \asymp_\times \|\vv\|$ for all $\vv\in\Lambda$. Here $\|\cdot\|$ denotes any norm on $V$.
\end{fact}

To prove the $\gtrsim$ direction of \eqref{CROSSONCTquasi2}, assume that $H$ is regular, fix $h = n(a,\vv)\in H$, and let $F$ be a finite generating set for $H$. Since $H$ is regular, the group $\Lambda = \pi(H) \leq \F^{d - 1}$ is discrete. Since $\F^{d - 1}$ is a finite-dimensional vector space, Fact \ref{factFDVS} guarantees the existence of a sequence $f_1,\ldots,f_n\in F$ such that $\pi(f_1\cdots f_n) = \pi(h)$ and $n \lesssim_\times \|\vv\|$. Let $f = f_1\cdots f_n$ and let $z = hf^{-1} \in \pi^{-1}(0) = Z$, say $z = n(b,\0)$. Applying \eqref{compositionlaw} and the second equation of \eqref{compositionlawbounds}, we see that $|b| \lesssim_\times |a| \vee \|\vv\|^2 \vee n^2 \lesssim_\times |a| \vee \|\vv\|^2$. On the other hand, applying Fact \ref{factFDVS} to $Z_\bp$ gives $\dist_Z(e,z) \lesssim_\times |b|$. Thus
\begin{align*}
0\vee 2\log\dist_H(z,h) \vee \log\dist_Z(e,z) &=_\pt 0\vee 2\log\dist_H(e,f) \vee \log\dist_Z(e,z)\\
&\lesssim_\plus 0\vee 2\log(n)\vee \log|b|\\
&\lesssim_\plus 0\vee 2\log\|\vv\| \vee \log(|a| \vee \|\vv\|^2)\\
&=_\pt 0\vee 2\log\|\vv\| \vee \log|a| = \dogo h.
\end{align*}
This completes the proof of \eqref{CROSSONCTquasi2}.

To prove the $\gtrsim$ direction of \eqref{CROSSONCTquasi1}, let $\cl H$ and $\cl Z$ be the Zariski closures of $H$ and $Z$ in $N_\bp$, respectively. Then $\cl H/\cl Z$ and $\cl Z$ are abelian Lie groups, and therefore isomorphic to finite-dimensional vector spaces. Let $\w\pi:\cl H\to\cl H/\cl Z$ be the projection map. Note that $\|\w\pi(n(a,\vv))\| \lesssim_\times |a| \vee \|\vv\|$ for all $n(a,\vv)\in H$. Here $\|\cdot\|$ denotes any norm on $\cl H/\cl Z$.

Since $\cl Z$ is a vector space, the fact that $Z$ is Zariski dense in $\cl Z$ simply means that $Z$ is a lattice in $\cl Z$. In particular, $Z$ is cocompact in $\cl Z$, which implies that $\w\pi(H)$ is discrete. Fix $h = n(a,\vv)\in H$, and let $F$ be a finite generating set for $H$. Then by Fact \ref{factFDVS}, there exists a sequence $f_1,\ldots,f_n\in F$ such that $\w\pi(f_1)\cdots\w\pi(f_n) = \w\pi(h)$ and $n\lesssim_\times \|\w\pi(h)\|\lesssim_\times |a| \vee \|\vv\|$. Let $f = f_1\cdots f_n$ and let $z = h f^{-1} \in H\cap \w\pi^{-1}(0) = H\cap \cl Z = Z$, say $z = n(b,\0)$. Applying \eqref{compositionlaw} and the second equation of \eqref{compositionlawbounds}, we see that $|b| \lesssim_\times |a| \vee \|\vv\|^2 \vee n^2 \lesssim_\times |a|^2 \vee \|\vv\|^2$. On the other hand, applying Fact \ref{factFDVS} to $\cl Z$ gives $\dist_Z(e,z) \lesssim_\times |b|$. Thus
\begin{align*}
0\vee\log\dist_H(e,h) &\leq_\pt 0\vee \log\dist_H(e,f) \vee \log\dist_Z(e,z)\\
&\lesssim_\plus 0\vee \log(n)\vee \log|b|\\
&\lesssim_\plus 0\vee \log(|a|\vee \|\vv\|) \vee \log(|a|^2\vee \|\vv\|^2)\\
&\asymp_\times 0\vee 2\log\|\vv\| \vee \log|a| = \dogo h.
\end{align*}
This completes the proof of \eqref{CROSSONCTquasi1}.
\end{proof}

\begin{corollary}
\label{corollaryCROSSONCTquasi}
Let $X$ and $\w X$ be finite-dimensional ROSSONCTs, let $H\leq\Isom(X)$ and $\w H\leq \Isom(\w X)$ be parabolic groups with fixed points $\bp$ and $\w\bp$, respectively, and let $\Phi:H\to\w H$ be an isomorphism. Then
\begin{itemize}
\item[(i)] \eqref{tukia} holds.
\item[(ii)] If $H$ and $\w H$ are regular, then \eqref{tukia2} holds if and only if $\Phi(Z)$ is commensurable to $\w Z$. Here $Z = Z(H)$ and $\w Z = Z(\w H)$.
\end{itemize}
\end{corollary}
\begin{proof}
\eqref{tukia} follows immediately from \eqref{CROSSONCTquasi1}. %, since
%\[
%0\vee \log\dist_H(e,h) \leq \min_{z\in Z} \big(0\vee 2\log\dist_H(z,h) \vee \log\dist_Z(e,z)\big) \leq 0\vee 2\log\dist_H(e,h).
%\]
Suppose that $H$ and $\w H$ are regular and that $\Phi(Z)$ is commensurable to $\w Z$. Since the right hand side of \eqref{CROSSONCTquasi2} depends on both $h$ and $Z$, let us write it as a function $R(h,Z)$. We then have
\[
\dogo h \asymp_\plus R(h,Z) = R(\w h, \Phi(Z)) \asymp_\plus R(\w h,\w Z) \asymp_\plus \|\w h\|.
\]
On the other hand, suppose that $\Phi(Z)$ and $\w Z$ are not commensurable. Without loss of generality, suppose that the index of $\Phi(Z)\cap \w Z$ in $\Phi(Z)$ is infinite. Since $\Phi(Z)$ is a finitely generated abelian group, it follows that there exists $\w h = \Phi(h)\in \Phi(Z)$ such that $\w h^n\notin \w Z$ for all $n\in\Z\butnot\{0\}$. Without loss of generality, suppose that $\w h\in \w H_2$; otherwise replace $h$ by an appropriate power. Then \eqref{corletteiozzi} implies that
\[
\|h^n\| \asymp_{\plus,h} \log(n) \text{ but } \|\w h^n\| \asymp_{\plus,h} 2\log(n).
\]
Thus \eqref{tukia2} fails along the sequence $(h_n)_1^\infty$.
\end{proof}

\begin{corollary}
\label{corollaryrealROSSONCTquasi}
In the context of Corollary \ref{corollaryCROSSONCTquasi}, if $X$ and $\w X$ are both real ROSSONCTs, then \eqref{tukia2} holds.
\end{corollary}
\begin{proof}
Since $\Im(\R) = \{0\}$, the group $Z_\bp$ is trivial and thus $Z$ and $\w Z$ are trivial as well; moreover, every discrete parabolic group is regular.
\end{proof}

\begin{corollary}
\label{corollarylatticequasi}
In the context of Corollary \ref{corollaryCROSSONCTquasi}, if we assume both that
\begin{itemize}
\item[(I)] $H$ is a lattice in $M_\bp N_\bp$, and that
\item[(II)] the underlying base fields $\F$ and $\w\F$ of $X$ and $\w X$ satisfy $\dim_\R(\F)\geq\dim_\R(\w\F)$,
\end{itemize}
then \eqref{tukia2} holds.
\end{corollary}
\begin{proof}
Let $H_2$, $\psi$, $H_3$, and $Z = Z(H)$ be as on page \pageref{lemmaCROSSONCTquasi}. Without loss of generality, we may assume that $H = H_3$ and $\w H = \w H_3$. Then $H$ is a lattice in $N_\bp$ and $\w H\leq \w N_{\w\bp}$.

Since $H$ is a lattice in $N_\bp$, $H$ is Zariski dense in $N_\bp$; this implies that $[H,H]$ is Zariski dense in $Z_\bp = [N_\bp,N_\bp]$. Thus, the rank of $[H,H]$ (and also of $\Phi([H,H]) = [\w H,\w H]$) is equal to $\dim_\R(\Im(\F)) = \dim_\R(\F) - 1$. Thus $\dim_\R(\w\F) - 1 = \rank([H,H]) \leq \dim(Z_\bp) = \dim_\R(\w\F) - 1$. Since by assumption $\dim_\R(\F)\geq\dim_\R(\w\F)$, equality holds. Thus $Z$ is a lattice in $Z_\bp$ and is commensurable to $[H,H]$. Similarly, $\w Z$ is a lattice in $\w Z_{\w\bp}$ and is commensurable to $[\w H,\w H]$. Thus, the groups $H$ and $\w H$ are regular. Finally, $\w Z$ is commensurable to $[\w H,\w H] = \Phi([H,H])$, which is commensurable to $\Phi(Z)$, so Corollary \ref{corollaryCROSSONCTquasi} finishes the proof.
\end{proof}

We can now prove Theorems \ref{theoremtukiaROSSONCT} and \ref{theoremxie} from the introduction:

\begin{proof}[Proof of Theorem \ref{theoremtukiaROSSONCT}]
If $G\leq\Isom(X)$ is a lattice, then every parabolic subgroup $G_\bp$ satisfies (I). Thus, combining Corollaries \ref{corollaryCROSSONCTquasi} and \ref{corollarylatticequasi} proves Theorem \ref{theoremtukiaROSSONCT}.
\end{proof}

\begin{proof}[Proof of Theorem \ref{theoremxie}]
Xie has observed that the main result of his paper generalizes to ROSSONCTs\ once one verifies that Tukia's isomorphism theorem and the Global Measure Formula both generalize to that setting (cf. \cite[p.1]{Xie}). We have just shown that Tukia's isomorphism theorem generalizes (to the present setting at least), and the Global Measure Formula has been shown to generalize by Schapira \cite[Th\'eor\`eme 3.2]{Schapira}.

Actually, we should mention a minor change that needs to be made to Xie's proof in the setting of ROSSONCTs: Since the Hausdorff and topological dimensions of the boundary of a nonreal ROSSONCT\ are not equal, at the top of \cite[p.252]{Xie} one should use Pansu's lemma \cite[Proposition 6.5]{Pansu1}, \cite[Lemma 2.3(a)]{Xie} to deduce the lower bound on the Hausdorff dimension of $\Lambda(G_2)$ (i.e. \cite[p.252, line 4]{Xie}) rather than using Szpilrajn's inequality between Hausdorff and topological dimensions (cf. \cite[p.252, lines 2-3]{Xie}).
\end{proof}

Note that in Xie's proof, quasisymmetry is used in an essential way due to his use of Pansu's lemma \cite[Corollary 7.2]{Pansu1}, \cite[Lemma 2.3]{Xie}. Thus, the fact that the stronger asymptotic \eqref{tukia2} holds in the context of Corollary \ref{corollarylatticequasi} is essential to the proof of Theorem \ref{theoremxie}. It remains to be answered whether Theorem \ref{theoremxie} holds if we drop the assumption $\dim_\R(\F) \geq \dim_\R(\w\F)$.

We end this section by giving an example of groups for which \eqref{tukia2} fails.

\begin{example}
\label{exampletukia}
Let $\H = \H_\C^3$, let $\bp = [(1,1,\0)]$, and define a homomorphism $\theta:\R^3\to N_\bp$ by $\theta(x,y,z) = n(xi ,(y,z))$, where $i = \sqrt{-1}$. Consider the parabolic groups $H,H',H''\leq N_\bp$ defined by
\begin{align*}
H &= \theta(\Z\times\Z\times\{0\})\\
H' &= \theta(\Lambda\times\{0\})\\
H'' &= \theta(\{0\}\times\Z\times\Z).
\end{align*}
In the middle equation, $\Lambda$ denotes a lattice in $\R^2$ that does not intersect the axes. Then the groups $H,H',H''$ are all isomorphic, but we will show below that \eqref{tukia2} cannot hold for any isomorphisms between them. This is accounted for in Corollary \ref{corollaryCROSSONCTquasi} as follows: The group $H'$ is irregular, so Corollary \ref{corollaryCROSSONCTquasi} does not apply; The groups $Z(H)$ and $Z(H'')$ are not almost isomorphic (the former is isomorphic to $\Z$ while the latter is isomorphic to $\{0\}$), so Corollary \ref{corollaryCROSSONCTquasi} does not apply.% The fact that \eqref{tukia2} does not hold for the isomorphisms between $H$, $H'$, and $H''$ is consistent with Corollary \ref{corollaryCROSSONCTquasi} for the following reasons:
\end{example}
\begin{proof}
Note that the function $\dogo\cdot$ is described on $\theta(\R^3)$ by
\[
\dogo{\theta(x,y,z)} \asymp_\plus 0\vee \log|x|\vee 2\log(|y|\vee|z|)
\]
(cf. \eqref{corletteiozzi}). Now let $h_1 = \theta((1,0,0)) \in H$, $h_2 = \theta((0,1,0))\in H$. Then
\[
\|h_i^n\| \asymp_\plus i \log(n);
\]
but if $\Phi$ is an isomorphism from $H$ to either $H'$ or $H''$, then
\[
\|\Phi(h_i)^n\| \asymp_\plus 2 \log(n).
\]
This demonstrates the failure of \eqref{tukia2}, as setting $h = h_1^n$ gives $\alpha_\bp = 2$ while setting $h = h_2^n$ gives $\alpha_\bp = 1$.

Next, let $\dist_{H'}$ and $\dist_{H''}$ be Cayley metrics on $H'$ and $H''$, respectively. Then for all $R\geq 1$,
\[
\sup_{\dist_{H'}(e,h') \leq R} \|h'\| \asymp_\plus 2\log(R) > \log(R) \asymp_\plus \inf_{\dist_{H'}(e,h') > R} \|h'\|.
\]
but
\[
\sup_{\dist_{H''}(e,h'') \leq R} \|h''\| \asymp_\plus \inf_{\dist_{H''}(e,h'') > R} \|h''\| \asymp_\plus 2\log(R)
\]
This demonstrates the failure of \eqref{tukia2} for any isomorphism between $H'$ and $H''$, as taking the supremum over a ball in the Cayley metric gives $\alpha_\bp = 1$, while taking the infimum over the complement of a ball in the Cayley metric gives $\alpha_\bp = 2$.
%; it follows that
%\[
%\zeta(x,y,z) := \lim_{n\to\infty} \frac{1}{\log(n)}\|\theta(x,y,z)\| = \begin{cases}
%1 & y = z = 0\\
%2 & \text{otherwise}
%\end{cases}.
%\]
%In particular, there exists $h_1,h_2\in H$ such that $\theta(h_i) = i$ (namely $h_1 = (1,0,0)$, $h_2 = (0,1,0)$). 
%
%
%then
%\begin{align*}
%\frac{\sup_H \zeta}{\inf_H \zeta} &= 2 > 1 = \frac{\sup_{H'}\zeta}{\inf_{H'}\zeta} = \frac{\sup_{H''}\zeta}{\inf_{H''}\zeta}\\
%\frac{\sup_{\R H} \zeta}{\inf_{\R H} \zeta} = \frac{\sup_{\R H''}\zeta}{\inf_{\R H''}\zeta} &= 2 > 1 = \frac{\sup_{\R H'}\zeta}{\inf_{\R H'}\zeta}\cdot
%\end{align*}
%But the quantities $\sup_\Lambda\zeta/\inf_\Lambda\zeta$ and $\sup_{\R\Lambda}\zeta/\inf_{\R\Lambda}\zeta$ are preserved under isomorphisms of lattices which satisfy \eqref{tukia2}.
\end{proof}

%\begin{proposition}
%With notation as in Theorem \ref{theoremtukia}, 
%
%\end{proposition}

%\begin{remark}
%The above proof actually shows more; namely, it shows that the equivariant boundary extensions of the isomorphisms between $H$, $H'$, and $H''$ are not quasisymmetric. Indeed, if $\phi:\Lambda\to\w\Lambda$ is a quasisymmetric equivariant boundary extension of an isomorphism $\Phi:H\to \w H$, then there exists an increasing homeomorphism $f:(0,\infty)\to(0,\infty)$ such that
%
%\end{remark}

\begin{remark}
\label{remarktukia}
The above proof actually shows more; namely, it shows that if $\Phi:G\to \w G$ is a type-preserving isomorphism so that for some $\bp\in\Lbp$, $G_\bp$ and $\w G_{\w\bp}$ are distinct elements of $\{H,H',H''\}$, then the equivariant boundary extension of $\Phi$ is not quasisymmetric.
\end{remark}
\begin{proof}
By contradiction suppose that the equivariant boundary extension $\phi:\Lambda\to\w\Lambda$ is quasisymmetric. Fix $\zeta\in \Lambda\butnot\{\bp\}$, and let $\w\zeta = \phi(\zeta)$. Then by equivariance, for each $h\in G_\bp$ we have
\[
\phi(h(\zeta)) = \w h(\w\zeta).
\]
Let $f:(0,\infty)\to(0,\infty)$ be as in Definition \ref{definitionquasisymmetric}, so that for all $\xi,\eta_1,\eta_2\in \Lambda$,
\[
\frac{\w\Dist(\w\xi,\w\eta_2)}{\w\Dist(\w\xi,\w\eta_1)} \leq f\left(\frac{\Dist(\xi,\eta_2)}{\Dist(\xi,\eta_1)}\right).
\]
Letting $\xi = \bp$ and $\eta_i = h_i(\zeta)$ gives
\[
\frac{\w\Dist(\w\bp,\w h_2(\w\zeta))}{\w\Dist(\w\bp,\w h_1(\w\zeta))} \leq f\left(\frac{\Dist(\bp,h_2(\zeta))}{\Dist(\bp,h_1(\zeta))}\right).
\]
But $\Dist(\bp,h_i(\zeta)) \asymp_{\times,\zeta} \Dist(\bp,h_i(\zero)) = e^{(1/2)\dogo{h_i}}$; thus
\[
\exp\left(\frac12 \left[\|\w h_2\| - \|\w h_1\|\right]\right) \leq f_2\exp\left(\frac12 \Big[\dogo{h_2} - \dogo{h_1}\Big]\right),
\]
where $f_2(t) = Cf(Ct)$ for some constant $C > 0$. Letting $f_3(t) = 2\log f_2(e^{(1/2)t})$ gives
\[
\|\w h_2\| - \|\w h_1\| \leq f_3(\dogo{h_2} - \dogo{h_1}).
\]
But this is readily seen to contradict the proof of Example \ref{exampletukia}.
\end{proof}

\appendix
\section{Relative hyperbolicity and geometrical finiteness}

While the results of this appendix are stated with the assumption that the maximal parabolic subgroups of the groups in question are finitely generated, it is possible that the arguments of \cite{Osin2} might be able to remove this hypothesis.

\begin{lemma}
\label{lemmarelativelyhyperbolic}
Let $X$ be a CAT(-1) space, let $G\leq\Isom(X)$ be a geometrically finite group, and suppose that the maximal parabolic subgroups of $G$ are finitely generated. Then $G$ is hyperbolic relative to the collection $\{\Stab(G;p) : p\in\Lbp(G)\}$, where $\Lbp(G)$ denotes the set of bounded parabolic points of $G$ and $\Stab(G;p)$ denotes the stabilizer of $p$ in $G$.
\end{lemma}
\begin{proof}
By \cite[Theorem 12.4.5]{DSU}, the limit set $\Lambda$ of $G$ is compact and consists entirely of conical points and bounded parabolic points.
\begin{claim}
\label{claimappendix1}
The action of $G$ on $\Lambda$ is a convergence action (i.e. it acts discretely on the space of triples of distinct points).
\end{claim}
\begin{subproof}
If $\xi_1,\xi_2,\xi_3\in\Lambda$ are distinct points such that $g_n(\xi_i)\to \xi_i$ for some sequence $(g_n)_1^\infty$ in $G$, then applying \cite[Lemma 7.4.2]{DSU} with $y_i^{(n)} = g_n(\xi_i)$ and $x_n = g_n(\zero)$ shows that $\lim_{n\to\infty} x_n \in \{\xi_1,\xi_2\}$. But there is some pair $i,j\in\{1,2,3\}$ such that $\lim_{n\to\infty} x_n \notin \{\xi_i,\xi_j\}$, so this is a contradiction.
\end{subproof}
\begin{claim}
\label{claimappendix2}
A point $\xi\in\Lambda$ is conical (resp. bounded parabolic) according to the definitions in \cite{Yaman} if and only if it is conical (resp. bounded parabolic) with respect to the definitions in \cite{DSU}.
\end{claim}
\begin{subproof}
For the equivalence of the definitions of bounded parabolic points, see \cite[Lemma 12.3.6(A)$\Leftrightarrow$(B)]{DSU}; the compactness of closed $\xi$-bounded subsets of $\Lambda\butnot\{\xi\}$ follows from \cite[Theorem 12.4.5]{DSU}.

If $g_n(\zero)\to \xi$ radially in the sense of \cite[Definition 7.1.2]{DSU}, then for all $\eta\neq\xi$, we have $\lb \w\eta | \w\xi \rb_\zero \asymp_\plus \lb \eta|\xi\rb_{g_n(\zero)} \asymp_\plus 0$, where $\w\eta = g_n^{-1}(\eta)$ and $\w\xi \df g_n^{-1}(\xi)$. Thus $\Dist(\w\eta,\w\xi) \gtrsim_\times 1$, so by selecting $\w\theta\in\Lambda$ so that $\Dist(\w\theta,\w\eta),\Dist(\w\theta,\w\xi) \gtrsim_\times 1$, we get $(\w\xi,\w\eta,\w\theta)\in \KK$ for some compact set $\KK\subset \{(\xi,\eta,\theta)\in \Lambda^3 \text{ distinct}\}$. Then $g_n(\w\theta)\to \xi$, so $\xi$ is a conical point in the sense of \cite[third Definition on p.62]{Yaman}.

Conversely, if $\xi$ is a conical limit point in the sense of \cite[third Definition on p.62]{Yaman}, then there exist $\eta\neq\xi$, a sequence $\theta_n\to \xi$, and a sequence $(g_n)_1^\infty$ in $G$ such that for all $n$, the distances between the points $\w\xi \df g_n^{-1}(\xi)$, $\w\eta \df g_n^{-1}(\eta)$, and $\w\theta \df g_n^{-1}(\theta)$ are bounded from below. Then $\lb \xi|\eta\rb_{g_n(\zero)} \asymp_\plus \lb \w\xi | \w\eta \rb_\zero \asymp_\plus 0$, so $g_n(\zero)$ is close to the geodesic connecting $\xi$ and $\eta$. Similarly, $g_n(\zero)$ is close to the geodesic connecting $\xi$ and $\theta_n$, so since $\theta_n\to\xi$ we have $g_n(\zero)\to\xi$. So $g_n(\zero)\to\xi$ radially in the sense of \cite[Definition 7.1.2]{DSU}.
\end{subproof}

If $G$ is nonelementary, then Claims \ref{claimappendix1} and \ref{claimappendix2} allow us to apply \cite[Theorem 0.1]{Yaman}, which completes the proof. (The  quotient $G\backslash\Lbp(G)$ is always finite if $G$ is geometrically finite \cite[Observation 12.4.12]{DSU}.)

Finally, if $G$ is elementary then the lemma is easily seen to hold, either because $G$ is parabolic (in which case it is trivially hyperbolic relative to the collection $\{G\}$) or because $G$ is elliptic or elementary loxodromic (in which case $G$ is hyperbolic relative to $\{\lb\id\rb\}$).
\end{proof}

\begin{theorem}
\label{theoremtukiaappendix}
Let $X$, $\w X$ be CAT(-1) spaces, let $G\leq\Isom(X)$ and $\w G\leq\Isom(\w X)$ be two geometrically finite groups, and let $\Phi:G\to \w G$ be a type-preserving isomorphism. Suppose that the maximal parabolic subgroups of $G$ and $\w G$ are finitely generated. Then there is an equivariant homeomorphism between $\Lambda := \Lambda(G)$ and $\w\Lambda := \Lambda(\w G)$.
\end{theorem}
\begin{proof}
The arguments of Lemma \ref{lemmarelativelyhyperbolic} show that \cite[Theorem 0.1]{Yaman} applies as long as $G$ and $\w G$ are nonelementary. In this case, $\Lambda$ and $\w\Lambda$ are equivariantly homeomorphic to the boundaries of $G$ and $\w G$ relative to the collections $\{\Stab(G;p):p\in\Lbp(G)\}$ and $\{\Stab(\w G;\w p) : \w p\in\Lbp(\w G)\}$, respectively. But since $\Phi$ is type-preserving, these two sets are equivariantly homeomorphic to each other.
\end{proof}

\bibliographystyle{amsplain}

\bibliography{bibliography}

\providecommand{\bysame}{\leavevmode\hbox to3em{\hrulefill}\thinspace}
\providecommand{\MR}{\relax\ifhmode\unskip\space\fi MR }
\providecommand{\MRhref}[2]{%
  \href{http://www.ams.org/mathscinet-getitem?mr=#1}{#2}
}
\providecommand{\href}[2]{#2}
\begin{thebibliography}{10}

\bibitem{Apanasov}
B.~N. Apanasov, \emph{Discrete groups in space and uniformization problems.
  {T}ranslated and revised from the 1983 {R}ussian original}, Mathematics and
  its Applications (Soviet Series), 40, Kluwer Academic Publishers Group,
  Dordrecht, 1991.

\bibitem{BeardonMaskit}
A.~F. Beardon and B.~Maskit, \emph{Limit points of {K}leinian groups and finite
  sided fundamental polyhedra}, Acta Math. \textbf{132} (1974), 1--12.

\bibitem{BHV}
M.~E.~B. Bekka, P.~de~la Harpe, and A.~Valette, \emph{{K}azhdan's property
  ({T})}, New Mathematical Monographs, 11, Cambridge University Press,
  Cambridge, 2008.

\bibitem{Bers}
L.~Bers, \emph{On boundaries of {T}eichm\"uller spaces and on {K}leinian
  groups. {I}}, Ann. of Math. (2) \textbf{91} (1970), 570--600.

\bibitem{BonkSchramm}
M.~Bonk and O.~Schramm, \emph{Embeddings of {G}romov hyperbolic spaces}, Geom.
  Funct. Anal. \textbf{10} (2000), no. 2, 266--306.

\bibitem{Bowditch_geometrical_finiteness}
B.~H. Bowditch, \emph{Geometrical finiteness for hyperbolic groups}, J. Funct.
  Anal. \textbf{113} (1993), no. 2, 245--317.

\bibitem{Bowditch_geometrical_finiteness2}
\bysame, \emph{Geometrical finiteness with variable negative curvature}, Duke
  Math. J. \textbf{77} (1995), no. 1, 229--274.

\bibitem{Bowditch_relatively_hyperbolic}
\bysame, \emph{Relatively hyperbolic groups}, Internat. J. Algebra Comput.
  \textbf{22} (2012), no. 3, 66 pp.

\bibitem{BridsonHaefliger}
M.~R. Bridson and A.~Haefliger, \emph{Metric spaces of non-positive curvature},
  Grundlehren der Mathematischen Wissenschaften, vol. 319, Springer-Verlag,
  Berlin, 1999.

\bibitem{CFKP}
J.~W. Cannon, W.~J. Floyd, R.~W. Kenyon, and W.~R. Parry, \emph{Hyperbolic
  geometry}, Flavors of geometry, Math. Sci. Res. Inst. Publ., vol.~31,
  Cambridge Univ. Press, Cambridge, 1997, pp.~59--115.

\bibitem{CCJJV}
P.-A. Cherix, M.~G. Cowling, P.~Jolissaint, P.~Julg, and A.~Valette,
  \emph{Groups with the {H}aagerup property: {G}romov's a-{T}-menability},
  Progress in Mathematics, vol. 197, Birkh\"auser Verlag, Basel, 2001.

\bibitem{Coornaert}
M.~Coornaert, \emph{Mesures de {P}atterson-{S}ullivan sur le bord d'un espace
  hyperbolique au sens de {G}romov ({P}atterson-{S}ullivan measures on the
  boundary of a hyperbolic space in the sense of {G}romov)}, Pacific J. Math.
  \textbf{159} (1993), no. 2, 241--270 (French).

\bibitem{CorletteIozzi}
K.~Corlette and A.~Iozzi, \emph{Limit sets of discrete groups of isometries of
  exotic hyperbolic spaces}, Trans. Amer. Math. Soc. \textbf{351} (1999), no.
  4, 1507--1530.

\bibitem{DSU}
T.~Das, D.~S. Simmons, and M.~Urba\'nski, \emph{Geometry and dynamics in
  {G}romov hyperbolic metric spaces: with an emphasis on non-proper settings},
  \url{http://arxiv.org/abs/1409.2155}, preprint 2014.

\bibitem{DrutuKapovich}
C.~Dru{\c t}u and M.~Kapovich, \emph{Lectures on geometric group theory},
  \url{https://www.math.ucdavis.edu/~kapovich/EPR/kapovich_drutu.pdf}.

\bibitem{Gromov3}
M.~Gromov, \emph{Hyperbolic groups}, Essays in group theory, Math. Sci. Res.
  Inst. Publ., vol.~8, Springer, New York, 1987, pp.~75--263.

\bibitem{Heintze}
E.~Heintze, \emph{On homogeneous manifolds of negative curvature}, Math. Ann.
  \textbf{211} (1974), 23--34.

\bibitem{Helgason}
S.~Helgason, \emph{Differential geometry, {L}ie groups, and symmetric spaces.},
  Pure and Applied Mathematics, 80, Academic Press, Inc. (Harcourt Brace
  Jovanovich, Publishers), New York-London, 1978.

\bibitem{Jorgensen}
T.~J{\o}rgensen, \emph{Compact 3-manifolds of constant negative curvature
  fibering over the circle}, Ann. of Math. (2) \textbf{106} (1977), no. 1,
  61--72.

\bibitem{Katok_book}
S.~R. Katok, \emph{Fuchsian groups}, Chicago Lectures in Mathematics,
  University of Chicago Press, Chicago, IL, 1992.

\bibitem{MackayTyson}
J.~M. Mackay and J.~T. Tyson, \emph{Conformal dimension: Theory and
  application}, University Lecture Series, 54, American Mathematical Society,
  Providence, RI, 2010.

\bibitem{Osin2}
D.~V. Osin, \emph{Relatively hyperbolic groups: intrinsic geometry, algebraic
  properties, and algorithmic problems}, Mem. Amer. Math. Soc. \textbf{179}
  (2006), no.~843, vi+100.

\bibitem{Pansu1}
P.~Pansu, \emph{M\'etriques de {C}arnot-{C}arath\'eodory et quasiisom\'etries
  des espaces sym\'etriques de rang un ({C}arnot-{C}arath\'eodory metrics and
  quasi-isometries of rank one symmetric spaces)}, Ann. of Math. (2)
  \textbf{129} (1989), no. 1, 1--60 (French).

\bibitem{Quint}
J.-F. Quint, \emph{An overview of {P}atterson--{S}ullivan theory},
  \url{www.math.u-bordeaux1.fr/~jquint/publications/courszurich.pdf}.

\bibitem{Ratcliffe}
J.~G. Ratcliffe, \emph{Foundations of hyperbolic manifolds}, Graduate Texts in
  Mathematics, vol. 149, Springer, New York, 2006.

\bibitem{Roblin1}
T.~Roblin, \emph{Ergodicit\'e et \'equidistribution en courbure n\'egative
  ({E}rgodicity and uniform distribution in negative curvature)}, M\'em. Soc.
  Math. Fr. (N. S.) \textbf{95} (2003), vi+96 pp. (French).

\bibitem{Schapira}
B.~Schapira, \emph{Lemme de l'ombre et non divergence des horosph\'eres d'une
  vari\'et\'e g\'eom\'etriquement finie. ({T}he shadow lemma and nondivergence
  of the horospheres of a geometrically finite manifold)}, Ann. Inst. Fourier
  (Grenoble) \textbf{54} (2004), no. 4, 939--987 (French).

\bibitem{Schwartz2}
R.~E. Schwartz, \emph{The quasi-isometry classification of rank one lattices},
  Inst. Hautes \'Etudes Sci. Publ. Math. \textbf{82} (1995), 133--168.

\bibitem{Tukia2}
P.~Tukia, \emph{On isomorphisms of geometrically finite {M}\"obius groups},
  Inst. Hautes \'Etudes Sci. Publ. Math. \textbf{61} (1985), 171--214.

\bibitem{Vaisala}
J.~V\"ais\"al\"a, \emph{{G}romov hyperbolic spaces}, Expo. Math. \textbf{23}
  (2005), no. 3, 187--231.

\bibitem{Xie}
X.~Xie, \emph{A {B}owen type rigidity theorem for non-cocompact hyperbolic
  groups}, Math. Z. \textbf{259} (2008), no. 2, 249--253.

\bibitem{Yaman}
A.~Yaman, \emph{A topological characterisation of relatively hyperbolic
  groups}, J. Reine Angew. Math. \textbf{566} (2004), 41--89.

\end{thebibliography}

\end{document}